\documentclass[12 pt]{amsart}
\usepackage[margin=1in]{geometry}
\usepackage{amssymb, amsmath, amsthm}
\usepackage{hyperref}
\usepackage{setspace}
\usepackage[all]{xy}

\newtheorem{theorem}{Theorem}[section]
\newtheorem{lemma}[theorem]{Lemma}
\newtheorem{proposition}[theorem]{Proposition}
\newtheorem{corollary}[theorem]{Corollary}
\newtheorem*{sublemma}{Sublemma}
\newtheorem{THM}{Theorem}
\newtheorem*{T1}{Theorem 1}
\newtheorem*{T2}{Theorem 2}

\theoremstyle{definition}
\newtheorem{remark}[theorem]{Remark}
\newtheorem{definition}[theorem]{Definition}
\newtheorem*{example}{Example}

\def\beq{\begin{eqnarray*}}
\def\eeq{\end{eqnarray*}}

\def\bQ{\mathbb{Q}}
\def\Q{\mathbb{Q}}
\def\bR{\mathbb{R}}
\def\R{\mathbb{R}}
\def\bZ{\mathbb{Z}}
\def\Z{\mathbb{Z}}

\def\bN{\mathbb{N}}
\def\N{\mathbb{N}}

\def\incl{\hookrightarrow}
\def\to{\rightarrow}

\def\tH{\widetilde{H}}

\def\eps{\varepsilon}

\def\dim{\mathrm{dim}\>}

\def\id{\mathrm{id}}
\def\x{\times}

\def\d{\partial}
\def\phi{\varphi}

\def\Hom{\mathrm{Hom}}

\def\Emb{\mathrm{Emb}}

\def\P{\mathcal{P}}

\def\L{\langle L \rangle}

\def\K{\mathcal{K}}

\def\E{E^\nu}
\def\Th{\mathrm{Th}}
\def\Bl{\mathrm{Bl}}
\def\map{\mathrm{map}}
\def\inte{\mathrm{int}}

\def\C{\mathcal{C}}

\title{A homotopy-theoretic view of Bott-Taubes integrals and knot spaces}
\author{Robin Koytcheff}
\date{\today}

\begin{document}
\maketitle


\begin{abstract}
We construct cohomology classes in the space of knots by considering a bundle over this space and ``integrating along the fiber" classes coming from the cohomology of configuration spaces using a Pontrjagin-Thom construction.  The bundle we consider is essentially the one considered by Bott and Taubes \cite{Bott-Taubes}, who integrated differential forms along the fiber to get knot invariants.  By doing this ``integration" homotopy-theoretically, we are able to produce \emph{integral} cohomology classes.  We then show how this integration is compatible with the homology operations on the space of long knots, as studied by Budney and Cohen \cite{Budney-Cohen}.  In particular we derive a product formula for evaluations of cohomology classes on homology classes, with respect to connect-sum of knots. 
\end{abstract}

\section{Introduction}
The space of knots, $\Emb(S^1, \R^3)$, is the space of smooth embeddings of $S^1$ into $\R^3$, and its connected components correspond to isotopy classes of knots.  Thus elements of $H^0 (\Emb(S^1, \R^3))=\Hom(\Z[\pi_0 \Emb(S^1, \R^3)], \Z)$ correspond to knot invariants.  More generally, classes in the cohomology of the space of knots can be thought of as naturally extending knot invariants, a subject of classical study.  The main contribution of this paper is to take a differential forms construction of cohomology classes in the knot space and recast it purely in terms of algebraic topology; use this to construct families of integral cohomology classes; and then show that these classes satisfy explicit product formulae with respect to connect-sums of knots.

In \cite{Bott-Taubes}, Bott and Taubes constructed knot invariants by considering a bundle over $\Emb(S^1,\R^3)$.  The fiber of this bundle is a compactification of a configuration space of $q+t$ points in $\R^3$, $q$ of which lie on the knot.
Bott and Taubes considered differential forms coming from the cohomology of configuration spaces, integrated them along the fiber of the bundle, and showed that the result is a zero-dimensional closed form, which thus represents a knot invariant.  This result concerned one particular knot invariant previously found through Chern-Simons theory.  However, the framework they set up was used by Thurston to construct a whole class of knot invariants for knots in $\R^3$ \cite{Thurston, VolicBT}, and by Cattaneo, Cotta-Ramusino, and Longoni to construct cohomology classes in $\Emb(S^1,\R^n)$ \cite{Cattaneo}.  The knot invariants constructed in \cite{Thurston} are Vassiliev invariants (i.e., finite type), and the graph cochain complex used in \cite{Cattaneo} to construct the cohomology classes is known to be quasi-isomorphic to the $E^1$ term of the Vassiliev spectral sequence.

In a rather different approach to studying knot spaces, Budney obtained results on the homotopy type and homology of knot spaces.  Let $\Emb(\R, \R^3)$ denote the space of long knots, i.e., the space of embeddings of $\R$ into $\R^3$ which agree with a standard embedding of $\R$ outisde of the unit interval.  Budney constructed a little 2-cubes operad action on a space which we call $\K$, which is homotopy equivalent to $\Emb(\R, \R^3)$, the space of long knots in $\R^3$. (He also constructed such an action on the space of framed long knots in $\R^n$.)  This can be viewed as a lifting of the connect-sum operation on isotopy classes of knots to a space-level operation on the space of knots; the fact that the little 2-cubes operad parametrizes it reflects its homotopy-commutativity.  Budney used this together with JSJ-decompositions of 3-manifolds and techniques of Hatcher's to further show that $\Emb(\R, \R^3)$ is the free 2-cubes object on the space of prime long knots \cite{Budney}.  Combining this with work of Fred Cohen \cite{FCohen} gave a computation of the homology of the space of long knots \cite{Budney-Cohen}.

In this paper, we consider a bundle $E_{q,t}$ similar to that of Bott and Taubes, but over the space $\K$ of ``fat long knots", as considered by Budney, in order to exploit the 2-cubes action as best as we can:
\[
 \xymatrix{
F_{q,t} \ar[r] & E_{q,t}\ar[d] \\
 & \K}
\]
The fiber $F_{q,t}$ is a compactification of the configuration space of $q+t$ points in $\R^3$, $q$ of which lie on the knot.  Let $C_q(\R^3)$ denote the compactified configuration space of $q$ points in $\R^3$.  We have a map $ev: E_{q,t}\to C_{q+t}(\R^3)$ such that the composition
\[
 \xymatrix{
F_{q,t} \ar[r] & E_{q,t} \ar[r] & C_{q+t}(\R^3)}
\]
is the obvious inclusion.  We carry out ``integration along the fiber", but by methods of algebraic topology.  Embedding the total space $E=E_{q,t}$ by a map $e_N$ into a Euclidean space of dimension $N$ and taking a Thom collapse (or pre-transfer) map \emph{roughly} gives a map
\[
\tau: \Sigma^N \K_+ \to E^{\nu_N}
\]
from the $N$-fold suspension of the base space $\K$ (union a disjoint basepoint) to the Thom space  of the normal bundle $\nu_N$ of $e_N$.  In cohomology this gives a map corresponding to integration along the fiber.  
However, the fibers in this case happen to be manifolds with corners, so we take some care to ensure that the embedding $e_N$ is \emph{neat}, i.e., that it preserves the corner structure.  In addition, we must quotient by the boundary to get a map as above from $\Sigma^N\K$, and the map we get is actually
\[
\tau: \Sigma^N \K_+ \to E^{\nu_N}/(\d E^{\nu_N}).
\]
By letting $N$ in $e_N$ approach $\infty$, we then get a map from the suspension \emph{spectrum} of $\K$ to the Thom \emph{spectrum} of the normal bundle to the total space, which induces in cohomology a map similar to the Bott-Taubes integration along the fiber.  It is worth noting that in the case of bona fide integration along the fiber, there is only a map on the level of forms, not cohomology, because of the presence of boundary.  Our method for constructing cohomology classes works equally well when $\K$ is replaced by the space of long knots in $\R^n$ or by $\Emb(S^1, \R^n)$, the space of closed knots in $\R^n$, for any $n\geq 3$.

Once we have set up the map $\tau$, we want to see how it behaves with respect to operations induced in homology by connect-sum, and more generally Budney's 2-cubes action on the long knot space.  We consider a space-level connect-sum $\mu:\K\x \K \to \K$ and define a multiplication 
\[\mu_C: C_q(\R^3)/\d C_q(\R^3) \x C_q(\R^3)/\d C_q(\R^3) \to C_{q+r}(\R^3)/\d C_{q+r}(\R^3)
\]
on the compactified configuration spaces modulo their boundaries.  This allows us to extend $\mu$ to a multiplication
\[ \mu_E: E_{q,t}/\d E_{q,t} \x E_{r,s} /\d E_{r,s} \to E_{q+r,t+s}/\d E_{q+r,t+s} 
\]
on the total space of our bundle modulo its boundary.  Careful scrutiny then shows that we can make the Thom collapse maps $\tau_{q,t}$ commute with the multiplications on the knot space and the total space.  To summarize, the first main result of the paper is
\begin{THM}
\label{mapofringspectra}
The total spaces $E_{q,t}$ of the bundles over the knot space have a multiplication $E_{q,t}/\d E_{q,t}\x E_{r,s}/\d E_{r,s}\to E_{q+r,t+s}/\d E_{q+r,t+s}$ which makes the wedge of Thom spectra $\bigvee_{q,t \in \bN} E^\nu_{q,t}/\d E^\nu_{q,t}$ into a ring spectrum.  (The precise definition of this spectrum is given in section \ref{thomringspectrumsection}.)  The Thom collapse maps for various $q,t$ induce a map of ring spectra
\[
\vee_{q,t} \tau_{q,t}: \bigvee_{q,t} \Sigma^\infty \K_+ \to \bigvee_{q,t} E^\nu_{q,t}/\d E^\nu_{q,t}
\]
where the multiplication in $\bigvee \Sigma^\infty \K_+$ comes from the space-level connect-sum.  Using the Thom isomorphism, this induces an ``integration along the fiber'' map in cohomology with arbitrary coefficients, producing classes in $H^*\K$.
\end{THM}

The compatibility of the multiplication on $\bigvee_{q,t} E^\nu_{q,t}/\d E^\nu_{q,t}$ with connect-sum on $\K$ allows us to derive the following product formula, which is the second main result of the paper:

\begin{THM}[Product Formula]
\label{productformula}
Let $\beta \in H^*(C_q(\R^3)/\d C_q(\R^3))$ and $a_1,a_2 \in H_*\K$.  Let $\theta_i$ and $\eta_i$ be classes in $H^*(\coprod_q C_q(\R^3) /\d C_q(\R^3))$ such that $\mu_C^* \beta =\sum_i \theta_i \otimes \eta_i$.  Then
\[
\langle \tau^* ev^* \beta, \> \mu_*(a_1 \otimes a_2)\rangle = \sum_i \langle \tau^* ev^* \theta_i,\> a_1\rangle \cdot 
\langle \tau^* ev^* \eta_i, a_2 \rangle,
\]
where the cohomology can be taken with coefficients in any ring.
\end{THM}

Moreover, the coproduct $\mu_C^*$ is computed in Lemma \ref{coprodCq} as the dual to an easily described product.  Thus for a particular $\beta$ we can determine the $\theta_i$ and $\eta_i$, and the above theorem will give us an explicit formula in terms of generators of $H^*(\coprod_q C_q(\R^3) /\d C_q(\R^3))$.

At this point, we can almost determine the evaluation of such $\alpha$ on any knot homology class from its evaluations on prime knot homology classes.  In fact, by Budney and Cohen's computation of the homology of the knot space, all that remains is to perform a similar calculation to the one above, but with the multiplication $\mu_E$ on $E/\d E$ replaced by a ``bracket" operation
\[
\C_2(2)\x E/\d E \x E/\d E \to E/\d E
\]
coming from a lift of the little 2-cubes action on $\K$ to $E/\d E$.  However, we will prove in Proposition \ref{no2cubeslift} that any multiplication on $\bigvee_{q,t} E_{q,t}/\d E_{q,t}$ compatible with the multiplication we define on $\bigvee_{q,t} C_{q+t}/\d C_{q+t}$ and the map 
\[
 ev: \bigvee_{q,t} E_{q,t}/\d E_{q,t} \to \bigvee_{q,t} C_{q+t}/\d C_{q+t}
\]
does not extend to a little 2-cubes action.  Now suppose more generally that one is given \emph{any} little 2-cubes actions on $\bigvee_{q,t} E_{q,t}/\d E_{q,t}$ and $\bigvee_{q,t} C_{q+t}/\d C_{q+t}$ compatible with the space-level connect-sum on $\K$ and the maps $ev$ and $\tau$.  We show in Proposition \ref{bracketprop} that the evaluation of a class $\tau^* ev^* \beta$ (as in Theorem \ref{productformula}) on the bracket of two classes in $H^*\K$ must be 0.

We note that even though we focus on the space of long knots in $\R^3$ because we study Budney's 2-cubes action on that space, Theorems \ref{mapofringspectra} and \ref{productformula} remain true when $\K$ is replaced by the space of long knots in $\R^n$ for any $n\geq 3$.

\subsection{Organization of the paper}
The rest of the paper is organized as follows.  Section \ref{background} describes the background on configuration space integrals from \cite{Bott-Taubes} and on homology of knot spaces from \cite{Budney} and \cite{Budney-Cohen}.  At the end of this section we define the bundle over the knot space that we will study.   Section \ref{xfer} focuses on the details of constructing the Thom collapse map, i.e., the ``integration along the fiber''.  This includes a review of a categorical approach to manifolds with corners, which we use to retain as much of the corner structure as possible.  In the first half of section \ref{mult}, we complete the proof of Theorem \ref{mapofringspectra}.  We do so by defining a multiplication on configuration spaces and hence on the total space of our bundle and then proving some lemmas about the compatibility of the multiplication with the Thom collapse maps.  In the second half of that section we prove Theorem \ref{productformula} by first examining the multiplication on configuration spaces in homology.  We then conclude with section \ref{bracket}, which contains two propositions related to the failure of the bracket operation to lift to the total space of our bundle.

\subsection{Acknowledgments}
The content of this paper will appear as part of the author's Stanford University Ph.D. thesis under the direction of Ralph Cohen.  The author would like to express deep gratitude to Ralph Cohen whose ideas, enthusiasm, and advice were indispensable for the completion of this article.  The author would also like to thank Nathan Habegger, Pascal Lambrechts, Paolo Salvatore, Dev Sinha, and Ismar Voli\'{c} for enlightening conversations relating to the subject matter of this paper.

\tableofcontents

\section{Background}
\label{background}

\subsection{The Bott-Taubes construction}
\label{BTsection}

Bott and Taubes constructed knot invariants by integrating differential forms along the fiber of a certain bundle over the knot space.  To describe this bundle, we first need to discuss configuration spaces.  For any space $X$, the ``open" (i.e., uncompactified) configuration space $C_q^0(X)$ of $q$ ordered points in $X$ is defined as the $q$-fold product minus the fat diagonal, i.e., 
\[C_q^0(X):=\{ (x_1,...,x_q)\in X^q | x_i\neq x_j \> \forall \> i\neq j\}.
\]
An embedding $X\incl Y$ induces a map of configuration spaces $C_q^0(X) \to C_q^0(Y)$.  

For a compact manifold $M$, the compactified configuration space $C_q(M)$ is defined as the closure of the image of the obvious embedding 
\begin{equation}
C_q^0(M) \incl M^q \x \prod_{\substack{S\subset \{1,...,q\}\\ |S|\geq 2}} \mathrm{Bl}(M^S, \Delta_S) \label{cptdefn}
\end{equation}
Here $M^S$ is the space maps from $S$ to $M$, a finite product of $M$'s, and $\Delta_S$ is the (thin) diagonal in $M^S$. $\Bl(X,Y)$ denotes the differential-geometric blowup of $X$ along $Y$, i.e., replace $Y$ by the sphere bundle of its normal bundle in $X$; alternatively, we can think of it as the complement of an open tubular neighborhood of $Y$ in $X$.  This compactification was first developed by Axelrod and Singer \cite{Axelrod-Singer}, who adapted work of Fulton and MacPherson \cite{Fulton-MacPherson} from the algebro-geometric setting to the differential-geometric setting.  The compactification leaves the homotopy type unchanged, so in particular, when we discuss homology of configuration spaces, we do not need to distinguish between the compactifications and their interiors.

Since $\R^n$ is noncompact, we must define $C_q(\R^n)$ as the pullback in the square below,
\[
\xymatrix{
C_q(\R^n) \ar[d] \ar[r] & C_{q+1}(S^n) \ar[d] \\
\{\infty\} \ar[r] & S^n }
\]
where the right-hand vertical map is projection onto the $(q+1)^{\textrm{th}}$ point.  That is, $C_q(\R^n)$ is the subspace of $C_{q+1}(S^n)$ where the $(q+1)^{\textrm{th}}$ point is at $\infty$.

Bott and Taubes \cite{Bott-Taubes} consider a fiber bundle whose total space $E_{q,t}$ is the pullback in the square below:
\[
\xymatrix{
&E_{q,t} \ar[r] \ar[d] & C_{q+t}(\R^3) \ar[d] \\
&\Emb(S^1, \R^3) \x C_q(S^1)_{conn} \ar[r] & C_q(\R^3) }
\]
In this diagram the space $C_q(S^1)_conn$ denotes one connected component of the compactified configuration space of $q$ points on the circle, while the spaces on the right side of the square are as defined above.  The lower horizontal arrow sends a knot $f$ and $q$ points on $S^1$ to the images of the $q$ points under $f$.  The right-hand map just projects to the first $q$ points.  The bundle that Bott and Taubes considered is 
\[
\xymatrix{
F_{q,t} \ar[r] & E_{q,t} \ar[d]^\pi \\ 
 &  \Emb(S^1, \R^3)}
\]
where $\pi$ is the left-hand map in the above square followed by the projection onto the first factor.  They showed that the fiber $F=F_{q,t}$ is a smooth (finite-dimensional) manifold with corners.  Cattaneo, Cotta-Ramusino, and Longoni considered the same bundle but with $\R^3$ replaced by $\R^n$ \cite{Cattaneo}.

To briefly review the main results of \cite{Bott-Taubes}, recall that we have maps $\phi_{ij}:C^0_{q+t}(\R^3) \to S^2$ given by 
\[
\phi_{ij}(x_1,...,x_q)=\frac{x_i - x_j}{|x_i - x_j|} 
\]
which extend to the compactifications $C_{q+t}(\R^3)$; if $\omega$ is a volume form on $S^2$, the images of the pullbacks $\omega_{ij} := \phi_{ij}^*\omega$ in cohomology generate $H^*C_{q+t}(\R^3)$ as an algebra \cite{Arnold, FCohen}.  Let $\theta_{ij}$ denote the pullbacks of the $\omega_{ij}$ to $\Omega_{dR}^* E_{q,t}$.  Recall that for any smooth fiber bundle $\xymatrix{F\ar[r] & E\ar[r]^-\pi & B}$, integration along the fiber is a map $\pi_*: \Omega_{dR}^p E \to \Omega_{dR}^{p-k} B$, where $k=\dim F$.  Bott and Taubes integrate (sums of) products $\theta$ of the $\theta_{ij}$ to get forms in $\Omega_{dR}^* \Emb(S^1, \R^3)$.  But because the fiber $F_{q,t}$ has boundary, it is not obvious that $\pi_*\theta$ is closed; in fact, Stokes' Theorem gives
\[
d\pi_* = \pi_* d + \pi^\d_*
\]
where $\pi^\d_*$ means ``restrict to $\d E_{q,t}$ and then integrate along $\d F_{q,t}$".  One main result of Bott and Taubes is a proof of the following, which had already been proven using physics-related techniques coming from Chern-Simons theory.
\begin{theorem}[Guadagnini, Martellini, and Mintchev \cite{Guadagnini}; Bar-Natan \cite{BNthesis}; Bott and Taubes \cite{Bott-Taubes}]
\[\pi_*\left(\frac{1}{4} \theta_{13} \theta_{24} - \frac{1}{3}\theta_{14}\theta_{24}\theta_{34}\right)
\]
is a closed form, i.e., an element $\in H^0(\Emb(S^1, \R^3))$. \qed
\end{theorem}

Other authors built upon their methods to prove more general results.  Thurston used them to show that any functional on chord diagrams can be integrated to a knot invariant (for knots in $\R^3$) \cite{Thurston}; Voli\'{c} provided more details on this in \cite{VolicBT}.  Cattaneo, Cotta-Ramusino, and Longoni constructed a graph cochain complex with a chain map to $\Omega_{dR}^* (\Emb(S^1, \R^n))$ inducing an injective map in cohomology \cite{Cattaneo}.  The integration along the fiber which we will carry out is very similar to that of Bott and Taubes, but done in a purely algebro-topological setting.

\subsection{Budney and Cohen's computation of the homology of the long knot space}

In recent work, Budney and Cohen computed the homology of the space of long knots.  This space, denoted $\Emb(\bR, \bR^3)$, is defined as the space of embeddings $f:\bR\to\bR^3$ which satisfy $f(t)=(t,0,0)$ for $|t|\geq 1$ and $f[-1,1]\subset [-1,1]\x D^2$.  This result followed from a result of Budney's \cite{Budney}, which says that the space of long knots is the free little 2-cubes object on the space of prime knots, combined with Cohen's computation of the homology of little 2-cubes objects $\C_2 X$ \cite{FCohen}.  We recall some of the details below.

Let $I=[-1,1]$.  Recall that the little 2-cubes operad is the operad that to each $n\in \bN$ associates the space $\C_2(n)$ of $n$ disjoint affine-linear embeddings of $I\x I \incl I\x I$.  We say $X$ has a $\C_2$-action if for each $n$ there is a map
\[
\C_2(n) \x_{\Sigma_n} X^n \to X
\]
satisfying certain associativity, equivariance, identity conditions (see \cite{goils} for the precise statements).  Each space $\C_2(n)$ can be thought of as parametrizing ways to multiply $n$ elements of $X$.  If $X$ has a $\C_2$-action, it has a multiplication which is homotopy-commutative, as can be seen by rotating two 2-cubes in the plane.  

In \cite{Budney}, Budney considered the space of embeddings $f: \bR\x D^2 \incl \bR\x D^2$ such that $f$ is the identity outside of $[-1,1]\x D^2$, which is homotopy equivalent to the space of {\it framed} long knots in $\R^3$.  He defined a $\C_2$-action on this space.  (More generally, he showed that the space of smooth embeddings $\R^n \x M \incl \R^n \x M$ which are the identity outside of $[-1,1]\x M$ has an action of the little $(n+1)$-cubes.)  The subspace $\K$ of all $f$, such that the linking number of $f|_{\bR\x(0,0)}$ with $f|_{\bR\x(0,1)}$ is 0, is homotopy equivalent to $\Emb(\bR, \bR^3)$ via $f\mapsto f|_{\bR\x (0,0)}$, and one can restrict the $\C_2$-action to $\K$.  The multiplication induced by this action on components of $\K$ corresponds to connect-sum of long knots, so this $\C_2$-action reflects the fact that connect-sum is commutative on $\pi_0\K$, which can be seen elementarily by pulling one knot through the other.  The paper \cite{Budney} contains figures clearly illustrating this fact, as well as the $\C_2$-action on the space of ``fat long knots" $\K$.  The main result therein is 
\begin{theorem}[Budney \cite{Budney}]
If $\K$ is the space of long knots in $\R^3$ (or homotopy equivalent to it) and $\P$ is the space of prime knots, then 
\[
\K \simeq \C_2( \P \sqcup\{*\}) := \coprod_{j=0}^\infty \C_2(j)\x_{\Sigma_j} \P^j. 
\]
\qed
\end{theorem}

To state the result on the homology of the space of long knots in $\R^3$, we need one definition:
\begin{definition} 
A {\it Gerstenhaber-Poisson algebra} $A$ is a graded-commutative $\Q$-algebra with a graded-skew-symmetric bilinear map 
\[
\{ \>, \>\}: A_m\otimes A_n \to A_{m+n+1}
\]
satisfying 
\begin{itemize}
\item[(1)]
Jacobi identity: $\{a,\{b,c\}\} = \pm \{\{a,b\},c\} \pm \{\{a,c\},b\}$
\item[(2)]
Leibniz rule: $\{a\cdot b, c\}=a\cdot \{b,c\} + (-1)^{|a||b|}b\cdot \{a,c\}$.  \qed
\end{itemize}
\end{definition}

Fred Cohen computed the homology of $\C_n X$ in terms of the homology of $X$ \cite{FCohen}, which combined with Budney's theorem above gives 
\begin{theorem}[Budney and Cohen \cite{Budney-Cohen}]
\begin{itemize}
\item[(1)]
$H_*(\K; \Q)$ is a free Gerstenhaber-Poisson algebra generated by $H_*(\P;\Q)$.
\item[(2)]
$H_*(\K; \Z/p)$ is a free ``restricted Gerstenhaber-Poisson algebra" generated by $H_*(\P; \Z/p)$. \qed
\end{itemize}
\end{theorem}

We are interested in these results because they indicate the possibility of calculating the evaluations of our classes in $H^*\K$ on arbitrary classes in $H_*\K$ in terms of evaluations of certain related cohomology classes on homology classes in $H_*\P$.


\subsection{Definition of our fiber bundle}

Now we can define the bundle we will consider.  
The total space of the bundle is the pullback below, where $C_q(\R)$ is  the compactified configuration space of $q$ points in $\R$, the spaces on the right are just compactified configuration spaces, and $\K$ is the space of fat long knots (and is homotopy equivalent to the space of long knots).  Since the interior of $\R\x D^2$ is homeomorphic to  $\bR^3$, we can take $C_q(\R\x D^2)$ to be the compactified configuration space $C_q(\bR^3)$ as defined by Bott and Taubes (i.e., subspace of $C_{q+1}(S^3)$).  The lower horizontal map below is given by restricting $f\in\K$ to $\bR\x(0,0)$ and then evaluating this embedding on the $q$ points. 
\begin{equation}
\xymatrix{
&E_{q,t} \ar[r] \ar[d] & C_{q+t}(\R\x D^2) \ar[d] \\
&\K\x C_q(\R) \ar[r] & C_q(\R\x D^2) 
}\label{Edefn}
\end{equation}
The bundle is given by 
\[
\xymatrix{
F_{q,t} \ar[r] & E_{q,t} \ar[d]^\pi \\ 
 &  \K}
\]
where $\pi$ is the left-hand vertical map in the square above followed by projection onto the first factor.  Thus a point in $E=E_{q,t}$ can be thought of as a fat (but not framed) long knot together with $q+t$ points in $\R\x D^2$ such that the first $q$ points are on the underlying long knot. We consider the bundle $\pi:E\to \K$ given by the top left-hand map above followed by projection.  So the fiber $F=F_{q,t}$ is the pullback below:
\[
\xymatrix{
F_{q,t} \ar[r] \ar[d] & C_{q+t}(\R\x D^2) \ar[d] \\
C_q(\R) \ar[r] & C_q(\R\x D^2) }
\]
$F$ is again a smooth manifold with corners, by the same argument as in the appendix of \cite{Bott-Taubes}.  \\ 


\section{The Pontrjagin-Thom map}
\label{xfer}
In this section, we will construct a Pontrjagin-Thom map (sometimes called the pre-transfer, or umkehr, map) for our bundle, which will induce a map in cohomology corresponding to integration along the fiber.  References for these maps in general include \cite{Cohen-Klein}, \cite{Becker-Gottlieb} and, at a more elementary level, \cite{Bredon}.  It will be natural to consider not just a map of spaces, but a map of spectra.  The entire section will be devoted to details of this construction.

\subsection{Manifolds with faces}
Our fiber $F$ is a manifold with corners, and in our construction we will try to retain as much of the corner structure as possible.  To do so, it is convenient to use the following results and terminology mostly from work of Laures \cite{Laures} to begin our construction of the pre-transfer map.  Let $\bR^{\L,N}$ denote $[0,\infty)^L\x \bR^N$ where possibly $N=\infty$.  Recall that a \emph{manifold with corners} is a space such that every point has a neighborhood diffeomorphic to some $\R^{\L,N}$.  We call the number of 0's in the coordinates at a point $x$ the \emph{codimension} $c(x)$ of $x$.  We define a {\it connected face} of $X$ as the closure of a component of $\{x|c(x)=1\}$.  
\begin{definition}
(J\"{a}nich)
$X$ is a {\it manifold with faces} or {\it $\L$-manifold} if 
\begin{itemize}
\item[(0)]
each $x\in X$ belongs to $c(x)$ different connected faces
\end{itemize}
and if we have (disjoint unions of connected) faces $(\partial_0 X, ..., \partial_{n-1}X)$ such that 
\begin{itemize}
\item[(1)] 
\[ \bigcup_{i=0}^{n-1} \partial_i X = \partial X\] and
\item[(2)]
$\forall \> i\neq j$, $\partial_i X \cap \partial_j X$ is a face of both $\partial_i X$ and $\partial_j X$.  \qed
\end{itemize}
\end{definition}

Note that not every manifold with corners can be given the structure of a manifold with faces.  For example consider a space homeomorphic to a circle, but pinched so that it has one corner.  At this point condition (0) above is not satisfied.

Laures described manifolds with faces using categorical language.   Let $\underline{2}$ denote the category $\{ \xymatrix{ 0 \ar[r] & 1} \}$.  We can think of an $\L$-manifold as a functor $X:\underline{2}^L \to \mathcal{T}op$, which we also call an $\L$-diagram of spaces or an $\L$-space.  In fact, for each $a\in \underline{2}^L$, set
\[
X(a) := \bigcap_{ \substack{i: a\leq (1,1,...,1,0,1,...,1) \\ 0 \mbox{ in } i^{\mathrm{th}} \mbox{ coord} } } \partial_i X.
\]
More generally, for any category $\C$ we will use the terminology $\L$-diagram (of e.g., groups, vector spaces, spectra) for a functor $\underline{2}^L \to \C$.  

\begin{definition}
A {\it neat embedding} $\xymatrix{X \ar@{^(->}[r]^-\iota & \R^{\L,N}:=[0,\infty)^L\x \R^N}$ is a natural transformation of functors so that 
\begin{itemize}
\item[(1)]
each $\iota(a)$ is an inclusion of a submanifold and
\item[(2)]
$\forall \> b<a$, the intersection of $\iota(X(a))$ with $\R^{\L,N}(b)$ is perpendicular.  \qed
\end{itemize}
\end{definition}

In particular, the normal bundle to a neat embedding is well defined.  The following gives a nice characterization of $\L$-manifolds, whose definition may at first seem unnatural. 

\begin{proposition}[Laures \cite{Laures}]
For any manifold with corners $X$, $X$ is an $\L$-manifold $\iff$ it can be neatly embedded into $\R^{\L,N}$ for some $N$.
\qed
\end{proposition}

\subsection{A neat embedding of our total space}
\label{neatembtotspace}

To proceed with a Pontrjagin-Thom construction, we need an embedding of the total space of finite codimension.  Of course $E\incl \K\x C_q(I) \x C_{q+t}(\R\x D^2)$ by definition, but since the $q$ points on $I$ are {\it embedded} by a knot in $\K$, we even have $E\incl \K \x C_{q+t}(\R\x D^2)$.  By Laures' proposition, the following guarantees that $C_{q+t}$ can be neatly embedded into some $\bR^{\L,N}$.

\begin{lemma}
$C_{q+t}:=C_{q+t}(\R\x D^2)=C_{q+t}(\bR^3)$ is a manifold with faces.
\end{lemma}

\begin{proof}
From \cite{Fulton-MacPherson} or \cite{Axelrod-Singer}, we know that each stratum of this space is labeled by $\mathfrak{S}=\{S_1,...,S_k\}$ where the $S_i$ are subsets of $\{1,...,q\}$ with $|S_i|\geq 2$ such that for any pair $S_i,S_j$, either $S_i\cap S_j=\emptyset$ or one is contained in the other.  Furthermore, $k$ is the codimension of the stratum $|\mathfrak{S}|$ labeled by $\mathfrak{S}$.  So we can partition the boundary into faces $|\{S_1\}|,...,|\{S_L\}|$, where all possible $S_i$ appear in the list.  Then a point $x \in |\mathfrak{S}|$ of codimension $k$ is contained in precisely $k$ connected faces, namely those labeled by $\{S_1\},\{S_2\},...,\{S_k\}$.  Moreover, the intersection of two faces $|\{S_i\}|\cap |\{S_j\}|$ is a codimension 2 stratum, hence a face of each of $|\{S_i\}|$ and $|\{S_j\}|$, thus proving the claim.  
\end{proof}

As explained in the appendix in \cite{Bott-Taubes}, the codimension of a stratum in the fiber $F=F_{q,t}$ is precisely the codimension of the corresponding stratum in $C_{q+t}$; furthermore, it is easily seen that the strata in $F$ correspond precisely to the strata in $C_{q+t}$.  Thus $F$ is also a manifold with faces, and $E$ is too, since all of its corner structure comes from $F$. Because strata in $F$ correspond to those in $C_{q+t}$, the composition $E\incl \K \x C_{q+t}(I\x D^2) \incl \K\x \bR^{\L,N}$ restricts on each fiber to a neat embedding, and hence we have a neat embedding 
\[
e_N:E\incl \K \x \bR^{\L, N}
\]
 of the total space.

\subsection{Thom collapse map}

Let $\nu_N = \nu_{e_N}$ be the normal bundle to the embedding $e_N$ defined above.  This normal bundle is well defined because $e_N$ is a neat embedding (or more precisely, because there is an appropriate notion of a collar for a neatly embedded $\L$-manifold, as in Lemma 2.1.6 of \cite{Laures}).  

%
%

The tubular neighborhood theorem holds for neatly embedded manifolds, so we can identify a tubular neighborhood of $E$ in $\bR^{\L,N} \x \K$ with $\nu_N$.  
Quotienting by its complement gives a Thom collapse map 
\[
\K \x \bR^{\L,N} \to (\K\x \bR^{\L,N} , \K \x \bR^{\L,N} - \nu_N),
\]
i.e., 
\[
\K \x \bR^{\L,N} \to E^{\nu_N},
\]
 where  $E^{\nu_N}=\mathrm{Th}(\nu_N \to E)$ is the Thom space of $\nu_N$.  Since $E$ locally looks like a product of $\K$ and $F$,  the $\L$-manifold stucture on $F$ makes $E$ into an (infinite-dimensional) $\L$-manifold.  The Thom space $E^{\nu_N}$ is not a manifold, but it is still an $\L$-space, i.e., a diagram of spaces indexed by $\underline{2}^L$.  In the usual Thom-Pontrjagin construction, one would take the one-point compactification of the Euclidean space on the left to get a sphere, but in this case, the corner structure requires us to examine the boundary more carefully.

Let $\bR^{\L,N}\cup \{\infty\}$ be the one-point compactification of $\bR^{\L,N}$.  This space is diffeomorphic as an $\L$-space to an iterated cone on a sphere, $C^L S^N$, since each space is diffeomorphic to the unit ball in $\bR^{\L, N}$ modulo its topological boundary.  When extending the above map to $\K \x (\bR^{\L, N} \cup \{\infty\})$, $\K\x\{\infty\}$ maps to the basepoint, so we get 
\begin{equation}
\K_+ \wedge C^L S^N \to E^{\nu_N} \label{mapspaces}
\end{equation}
(where the subscript ``+" denotes disjoint union with a basepoint).  Let $\partial X$ denote the complement of the top stratum of any $\L$-space $X$.  The Thom collapse map induces
\begin{equation}
(\K_+ \wedge C^L S^N, \K_+ \wedge \partial(C^L S^N)) \to (E ^{\nu_N}, \partial E ^{\nu_N}) \label{mappairs}.
\end{equation}
which is 
\begin{equation}
\tau: \Sigma^{L+N} (\K_+) \to E^{\nu_N} / \partial E ^{\nu_N} \label{quotientspaces}
\end{equation}
(where $\Sigma X$ denotes the reduced suspension $S^1 \wedge X$).  Note that we get such a $\tau=\tau_{q,t}$ for each $q,t$, since $E=E_{q,t}$ depends on $q,t$.  

\begin{remark}
In ($\ref{mappairs}$), we could replace the ``$\partial$" terms on either side these by smaller unions of filtrations to get analogous maps for appropriate subsets of $\underline{2}^L$.  Furthermore, instead of considering on either side the spaces indexed by $(1,1,...,1)$ modulo their boundaries, we could consider the spaces indexed by some smaller $a\in \underline{2}^L$ modulo their boundaries to give a map from some (lower) suspension of $\K$ to part of the boundary of $E^{\nu_N}$ modulo its boundary; this would correspond to integration along the fiber on a subspace of the fiber where some of the points have collided.
\qed
\end{remark}

\subsection{A map of spectra}

We can choose the embeddings $e_N$ to be compatible with each other in that the following square commutes:
\[
\xymatrix{
E \ar@{^(->}[r]^-{e_N} \ar[d]_{\mathrm{id}} & \K \x \bR^{\L,N}\ar[d]^{\mathrm{id} \x \mathrm{inclusion}} \\
E \ar@{^(->}[r]^-{e_{N+1}} & \K \x \bR^{\L,N} \x \R }
\]
This gives us for each $N$ an isomorphism $\nu_N \oplus \eps^1 \cong \nu_{N+1}$.  This induces a map of Thom spaces
\[
\Sigma E^{\nu_N} \to  E^{\nu_{N+1}},
\]
and we can thus define a Thom spectrum $\E$ by $(\E)_{L+N} = E^{\nu_N}$ for sufficiently large $N$. Notice that the image under the above map of $\partial E^{\nu_N}$ is contained in $\partial E^{\nu_{N+1}}$, so there is a subspectrum $\partial \E$.  Let $\Sigma^{\infty} \K_+$ be the suspension spectrum of $\K$, and let $\Sigma^{\infty - L}\K_+$ be the spectrum whose $N^{\textrm{th}}$ space is $\Sigma^{N-L}\K_+$.  We have now set up the framework needed to prove the following:

\begin{theorem}
\label{mapofspectra}
The maps (\ref{mapspaces}) and (\ref{quotientspaces}) induce maps of spectra
\begin{equation}
\widetilde{\tau}: C^L\Sigma^{\infty-L} \K_+ \to \E \label{mapspectra}
\end{equation}
and 
\begin{equation}
\tau: \Sigma^{\infty}\K_+ \to \E/ \partial \E \label{quotientspectra}.
\end{equation}
\end{theorem}

\begin{proof}
We check that the diagram below commutes, where it should be fairly obvious what all the maps are.  In particular, the horizontal maps induce the structure maps of the spectra.
\[
\xymatrix{
\K \x \bR^{\L,N} \x \bR \ar[r] \ar[d] & \K \x \bR^{\L, N+1} \ar[d] \\
\Th(\nu_N \to  E) \x \bR \ar[d] & \Th(\nu_{N+1}\to E) \ar[d] \\
\Th(\nu_N \oplus \eps^1 \to E)\ar[r]  &  \Th(\nu_{N+1} \to E)  }
\]
The diagram commutes because the left side is just the restriction of the Thom collapse map on the right side and because the horizontal maps are the appropriate inclusions. 
We get a similar diagram on the appropriate boundary subspaces, and thus the map of spectra.
\end{proof}

We note that the proof of Theorem \ref{mapofspectra} did not rely on any properties of knots in $\R^3$, so it holds when $\K$ is replaced by $\Emb(\R, \R^n)$, the space of long knots in $\R^n$, for any $n\geq 3$.  It also does not rely on the knots' being long, so it holds when $\K$ is replaced by $\Emb(S^1, \R^n)$, the space of closed knots in $\R^n$, for $n\geq 3$.  In fact, in that case the construction would proceed more like the original Bott-Taubes construction \cite{Bott-Taubes, Cattaneo}.  We leave it to the reader to make the straightforward adjustments needed to prove the various different versions of this Theorem.

\subsection{Induced maps in (co)homology}

The theorem above gives us maps in (co)homology which correspond to the integration along the fiber of Bott and Taubes:

\begin{corollary}
\label{mapincoh}
We have a map in homology 
\[
\tau_*: H_*(\K) \to H_{*+(n+qt)}(E, \partial E).
\]
and similarly a map in cohomology
\[
\tau^*: H^*(E/ \partial E) \to H^{*-(n+qt)}(\K)
\]
corresponding to the integration along the fiber of Bott and Taubes.  On the level of spectra, these maps are 
\[
\tH_*(\Sigma^\infty \K_+) \to \tH_{*+(n+qt)} (\E/ \partial \E)
\]
and 
\[
\tH^*(\E/\partial \E) \to \tH^{*-(n+qt)} (\Sigma^\infty \K_+)
\].
\end{corollary}

\begin{proof}
The suspension isomorphism followed by the map in homology induced by (\ref{quotientspaces}), and subsequently by the relative Thom isomorphism in homology gives 
\[
H_{*-(L+N)} (\K) \cong \tH_*(\Sigma^{L+N}((\K)_+)) \to \tH_*(E^{\nu_N} / \partial E^{\nu_N})\cong H_*(E^{\nu_N}, \partial E^{\nu_N}) \cong H_{*-k}(E, \partial E)
\]
where $k$ is the dimension of $\nu_N$, i.e., the codimension of the embedding, which is $L+N - \dim F= L+N - (n+qt)$.  A similar sequence of maps in cohomology gives $\tau^*$.
\end{proof}

As with Theorem \ref{mapofspectra}, Corollary \ref{mapincoh} remains true when $\K$ is replaced by $\Emb(\R, \R^n)$ or $\Emb(S^1, \R^n)$ for any $n\geq 3$.


\section{Multiplicative structure}
\label{mult}

Now that we have constructed the pre-transfer map, we have made some progress towards proving Theorem \ref{mapofringspectra}, which is a strengthening of Theorem \ref{mapofspectra}, and whose proof will be completed in this section.  We start by defining a multiplication on $\coprod_q C_q(\R^3)/ \d C_q(\R^3)$ in section \ref{multconf}.  This is used to define a multiplication on $\coprod_{q,t} E_{q,t}/\d E_{q,t}$, whose compatibility with our neat embeddings and the resulting Thom collapse maps is proven in several lemmas in section \ref{multtotalspace}.  In section \ref{thomringspectrumsection}, we finish the proof of Theorem \ref{mapofringspectra}.

Then having determined the compatibility of our map with certain multiplicative structures, we exploit it to prove Theorem \ref{productformula}.   This is done in section \ref{productformulasection}, after computing the multiplication on configuration spaces in (co)homology in section \ref{multconfhomology}.  Section \ref{bracket} discusses the failure of the bracket operation to lift to the total space of our bundle.

%
%

\subsection{Multiplication on configuration spaces}
\label{multconf}

In \cite{Budney}, Budney defined an action of the little 2-cubes operad $\C_2$ on $\K$ which extends the connect-sum operation on $\pi_0(\K)$.  For more information on the little cubes operad or operads in general, consult \cite{goils}.  In trying to extend this action to the total space, it is clear geometrically that we must now consider a disjoint union $\coprod_{q,t\in \N} E_{q,t}$ (where $E_{0,0}\cong \K$).  Unfortunately, we do not extend the $\C_2$-action to one on $\coprod_{q,t} E_{q,t}$ (or even $\coprod_{q,t} E_{q,t}/\partial E_{q,t}$).  Yet we will extend a space-level connect-sum to a multiplication on $\coprod_{q,t} (E_{q,t}/\partial E_{q,t})$ by defining a multiplication on $\coprod_q (C_q(\R\x D^2)/\partial C_q(\R\x D^2))$ which is compatible with the map induced by $\xymatrix{E_{q,t}\ar[r]^-{ev} & C_{q+t}(\R\x D^2)}$ on the quotients by boundaries.  We will compute the multiplication on configuration spaces in (co)homology in section \ref{multconfhomology} and then use it to show that the multiplication on the total space does not extend to a $\C_2$-action in section \ref{bracket}.

We start by endowing the space $\coprod_q (C_q(\R\x D^2)/\partial C_q(\R\x D^2))$ with an action of the ``nonsymmetric little intervals operad", i.e., an action of the little intervals $\C_1$, except that the equivariance condition is not satisfied; here the disjoint union is taken over the nonnegative integers, and $C_0(\R\x D^2)$ is just a point for the empty configuration.  To simplify notation, we will often abbreviate $C_q:=C_q(\R \x D^2)$ and $X/\d:=X/\d X$ for any space $X$.  

For $\ell\in \C_1(1)$, we will also use $\ell$ to denote the unique extension of this little interval to an 
affine-linear map $\R\to \R$.  

\begin{definition}
\label{intactiononCq}
Define the action of one little interval on the interior of $C_q/\partial$ as follows.  Let $\mathbf{x}=(x_1,...,x_q) \in \mathrm{int}\>C_q, q>0$, and set
\[
\ell\cdot \mathbf{x} = ((\ell\x \id)(x_1),...,(\ell\x \id)(x_q)),
\]
and for the point $*\in C_0$, set $\ell \cdot * = *$.  Now define the action on the interior of $\coprod_q (C_q/\partial)$ by
\begin{align*}
\C_1(j) \x \left(\coprod_q (C_q/\d)\right)^j &\to \coprod_q (C_q/\partial) \\
\mbox{by } (\ell_1,...,\ell_j;\mathbf{x}^1,...,\mathbf{x}^j) &\mapsto (\ell_1\cdot \mathbf{x}^1, \ell_2\cdot \mathbf{x}^2, ..., \ell_j\cdot \mathbf{x}^j)
\end{align*}
If $\mathbf{x}^i$ is the empty configuration $* \in C_0$, we just delete it from the right-hand side.  This defines the map on the interior, and we conclude the definition by sending 
\[
(\ell_1,...\ell_j; \mathbf{x}^1,...,\mathbf{x}^j) \mapsto * \in C_{q_1+...+q_j}/\d \mbox{ if } \mathbf{x}^i \in\d C_{q_i} \mbox{ for some } i.
\]
\qed
\end{definition}

Thus if each $\mathbf{x}^i \in C_{q_i}$, then the right-hand side is in $C_{q_1+...+q_j}$.  A choice of one object $\ell\in \C_1(2)$ induces a multiplication on our space which, because of the little intervals action, is homotopy-associative.  Finally notice that if we had each little interval $\ell$ act trivially on a configuration of points, we would get a multiplication homotopic to the one we have defined.

\begin{remark} 
\label{rmkmultfaces}
It is easy to see that the action (and in particular the induced multiplication) defined above is valid not only on $\coprod_q (C_q/\d)$ but also on faces of the compactified configuration spaces modulo their boundaries.  From the stratification of the $C_q$'s \cite{Fulton-MacPherson, Axelrod-Singer} recalled in section \ref{neatembtotspace}, we see that if $|\mathfrak{S}|$ and $|\mathfrak{T}|$ are strata of codimensions $k$ and $l$ respectively, then we have a map 
\[
|\mathfrak{S}|/\d \x |\mathfrak{T}|/\d \to |\mathfrak{U}|/\d
\]
where $|\mathfrak{U}|$ is a codimension $k+l$ stratum.  This is because if $\mathfrak{S}=\{S_1,...,S_k\}$ and $\mathfrak{T}=\{T_1,...,T_l\}$, the image is in the stratum corresponding to $\mathfrak{U}=\{S_1,...,S_k, T_1,...,T_l\}$.
\qed
\end{remark}

\begin{remark} \label{nonsymmC2} Note that we are not using the fact that the intervals that make up an object $\ell \in \C_1$ are disjoint.  Therefore we could equip $\coprod_q (C_q/\d)$ with a ``nonsymmetric" $\C_2$-action (or even ``nonsymmetric" $\C_n$ action for any $n$) obtained by just projecting to the first coordinate of the 2-cubes.  However, because it is ``nonsymmetric", the induced multiplication need \emph{not} be homotopy-commutative.
\qed
\end{remark}

In order to have a product structure which is compatible with the map $ev: E_{q,t}/\d \to C_{q+t}/\d$, we need to make one modification to Definition \ref{intactiononCq}.    Recall that a point in $E_{q,t}$ is a fat long knot $f$ together with points $(x_1,...,x_{q+t})$ such that $x_1,...,x_q$ lie on $f|_{\R\x(0,0)}$.  We want to define the multiplication 
\begin{align*}
E_{q,t}/\d \x E_{r,s}/\d &\to E_{q+r,t+s}/\d  \\
\mbox{by }
((f_1; x_1,...,x_{q+t}),(f_2;y_1,...,y_{r+s})) &\mapsto 
(\mu(f_1, f_2); x_1,...,x_q,y_1,...,y_r,x_{q+1},...,x_{q+t},y_{r+1},...,y_{r+s})
\end{align*}
for interior points.  So to ensure compatibility with $ev$, we need to keep track of which points in $C_{q+t}/\d$ are on the knot, and we permute points in the multiplication
\[
C_{q+t}/\d \x C_{r+s}/\d \to C_{q+t+r+s}/\d
\]
so that the first $q$ and first $r$ points in the left-hand factors become the first $q+r$ points on the right-hand side.  For this reason define 
\[
C_{q,t} := C_{q+t}=C_{q+t}(\R\x D^2).
\]
Thus $C_{q,t}$ and $C_{q+t}$ are identical spaces, but the point of the (slightly!) different notation is that the multiplications on them will be different.  In full generality, we endow $\coprod_{q,t} (C_{q,t}/\d)$ with not just a multiplication but an action of the nonsymmetric little intervals operad; it is the same as that on $\coprod_q (C_q/\d)$ except that we reorder the points so that those on the knot come first.

\begin{definition}
\label{modintactiononCq}
Let one little interval act on a point in $C_{q,t}/\d$ as before, and given 
\[\mathbf{x}^i = (x^i_1,...,x^i_{q_i+t_i}) \in C_{q_i,t_i}, i=1,...,j
\]
define 
\begin{align*}
 \C_1(j) \x \left(\coprod_{q,t} (C_{q,t}/\d)\right)^j \to & \coprod_{q,t} (C_{q,t}/\d) \\
\mbox{by } (\ell_1,...,\ell_j;\mathbf{x}^1,...,\mathbf{x}^j) \mapsto &
(\ell_1\cdot(x^1_1,...,x^1_{q_1}), \ell_2\cdot(x^2_1,...,x^2_{q_2}),...,\ell_j\cdot(x^j_1,...,x^j_{q_j}), \\ 
& \>\ell_1\cdot(x^1_{q_1+1},...,x^1_{q_1+t_1}),,...,\ell_j\cdot(x^j_{q_j+1},...,x^j_{q_j+t_j})).
\end{align*}
As before, the empty configuration in $C_{0,0}$ can be erased, and we send boundary basepoint to boundary basepoint.
\qed
\end{definition}

We are mainly just interested in the induced multiplication
\[
C_{q,t}/\d \x C_{r,s}/\d \to C_{q+t, r+s}/\d
\]
and it is useful to note that this can be expressed as the composition
\[
\xymatrix{
C_{q+t}/\d \x C_{r+s}/\d \ar[r] & C_{q+t+r+s}/\d \ar[r]^\sigma & C_{q+t+r+s}/\d
}
\]
of the multiplication on $\coprod_q (C_q/\d)$ followed by the diffeomorphism induced by the permutation $\sigma$ which shifts the $r$ points to the left of the $t$ points.  That is, 
\[
\sigma:\{1,...,q+t+r+s\}\to \{1,...,q+t+r+s\} 
\]
\begin{equation}
\label{sigmadefn}
\sigma(i)=\left\{\begin{array}{ll}
i & \mbox{ if } i\in[1,q]\cup [q+t+r+1, q+t+r+s] \\
i-t & \mbox{ if } i\in[q+t+1, q+t+r] \\
i+r & \mbox{ if } i\in[q+1, q+t] 
\end{array} \right. 
\end{equation}
As in  remark \ref{rmkmultfaces}, this multiplication can also be defined on strata of the $C_q$ modulo their boundaries.  As in remark \ref{nonsymmC2}, we could equip $\coprod_q (C_q/\d)$ with a ``nonsymmetric" $\C_2$-action, but again, it is not so interesting to consider because the homotopy it gives is not between $\mathbf{x}\cdot \mathbf{y}$ and $\mathbf{y}\cdot \mathbf{x}$.

\subsection{Multiplication on the total space and its Thom space}
\label{multtotalspace}

There is a $\C_1$-action on $\K$ obtained by projecting Budney's $\C_2$-action to the first coordinate of the 2-cubes, which can be considered as a nonsymmetric $\C_1$-action by forgetting the symmetric group action on the intervals.  Together with definition \ref{modintactiononCq}, this gives an action on the product $\K \x \coprod_{q,t} (C_{q,t}/\d)$.  Recall that $E_{q,t} \incl \K \x C_{q+t}$, and because all of the boundary structure in $E_{q,t}$ comes from $C_{q,t}$, we also have $E_{q,t}/\d \incl \K \x (C_{q,t}/\d)$.  The following lemma implies that $\coprod_{q,t} E_{q,t}/\partial$ has a multiplication which is associative up to homotopy.
\begin{lemma}
The nonsymmetric $\C_1$-action on $\K \x \coprod_{q,t} (C_{q,t}/\d)$ restricts to a nonsymmetric $\C_1$-action on $\coprod_{q,t} (E_{q,t}/\d)$, i.e., a map for each $q,t,r,s$
\[
E_{q,t}/\d \x E_{r,s}/\d \to E_{q+r,t+s}/\d .
\]
\end{lemma}
\begin{proof}
The only thing to check is that the first $q+r$ points on the right-hand side lie on the connect-sum knot.  But each little interval only alters each embedding in the long direction and because of our definition of the multiplication on $\coprod_{q,t}(C_{q,t}/\d)$, the points on the knot are listed first, as required.  We cannot ensure that this multiplication puts the $q+r$ points on $\R$ in order on the connect-sum knot, but since we considered all connected components of $C_q(\R)$ (rather than just one) in our definition of $E_{q,t}$, the map above is in fact well defined.
\end{proof}



We now want to show that the Thom collapse map respects this multiplication.  Fix a nonsymmetric little intervals object $\vec{\ell}=(\ell_1, \ell_2)$ so as to get multiplications $\mu, \mu_C$ and $\mu_E$ on the spaces $\K, \coprod_q (C_q/\d)$, and $\coprod_{q,t} (E_{q,t}/\d)$.  Of course any two choices of $\vec{\ell}$ give homotopic multiplications.  As before, let each $\ell_i$ also denote the affine-linear map $\R\to \R$ to which it corresponds.  


\begin{lemma}
\label{embmultcommutesonefactor}
Fix $q,t,r,s\in \bN$, and let $M'+M''=M$.  We have the commutative diagram below
\begin{equation}
\xymatrix{
E_{q,t}/\partial \x E_{r,s}/\partial \ar[r]^-{\mu_E} \ar@{^(->}[d] & E_{q+r, t+s}/\partial \ar@{^(->}[d] \\
\K \x C_{q,t}/\partial  \x \K \x C_{r,s}/\partial \ar@{^(->}[d]\ar[r]^-{\mu \x \mu_C} & \K \x C_{q+r,t+s}/\partial \ar@{^(->}[d]\\
\K \wedge S^{M'} \x \K \wedge S^{M''} \ar[r] & \K \wedge S^M } \label{embedonefactor}
\end{equation}
which is induced by the following diagram, where $L'+L''=L$, $N'+N''=N$, and where the dotted arrows are the restrictions of the maps $\mu$ in (\ref{embedonefactor}) to the preimages of the interiors, denoted by $U$ and $V$ (so the maps are not defined on the whole spaces on the left):

\begin{equation}
\xymatrix{
E_{q,t} \x E_{r,s} \ar@{-->}[r]^-{\mu_E|U} \ar@{^(->}[d] & E_{q+r, t+s} \ar@{^(->}[d] \\
\K \x C_{q,t} \x \K \x C_{r,s} \ar@{^(->}[d]\ar@{-->}[r]^-{(\mu\x \mu_C)|V} & \K \x C_{q+r,t+s} \ar@{^(->}[d]\\
\K \x \bR^{\langle L'\rangle,N'} \x \K \x \bR^{\langle L'' \rangle, N''} \ar[r] & \K \x\bR^{\L, N} } \label{embedonefactorinterior}
\end{equation}
\end{lemma}

\begin{proof}
We start by more closely examining the neat embeddings of the compactified configuration spaces to check that we can make them compatible with the multiplication defined on their quotients and with the obvious ``multiplication" on the Euclidean spaces, i.e., $\R^{N'} \x \R^{N''} \to \R^{N'+N''}$.  Recall that for compact $X$, $C_q(X)$ is defined as the closure of the image of the ``open" configuration space under the obvious map 
\[
C_q^0(X) \incl  X^q \x \prod_{\substack{S\subset \{1,...,q\}\\ |S|\geq 2}} \Bl(X^S, \Delta_S)
\]
and that $C_q:=C_q(\R\x D^2)\cong C_q(\R^3)$ is the subspace of $C_{q+1}(S^3)$ where the $(q+1)^{\textrm{th}}$ point is at $\infty$.  It is easily seen that we can embed
\[
C_q(\R^3)\incl (S^3)^q \x \prod_{\substack{S\subset \{1,...,q+1\}\\ |S|\geq 2}} \Bl((S^3)^S, \Delta_S).
\]

\begin{sublemma}
\label{multCqembs}
We commutative diagrams for all $q,r\in \N$ as below, where the top map is only defined on the preimage of the interior under the multiplication on the quotients by boundary.  We consider $\infty$ to be the $(q+r+1)^{\textrm{th}}$ point.
\[
\xymatrix{
C_q(\R\x D^2)\x C_r(\R\x D^2) \ar@{-->}[r] \ar@{^(->}[d] & C_{q+r}(\R\x D^2) \ar@{^(->}[d] \\
(S^3)^q \x \underset{S\subset\{1,...,q, q+r+1\}, |S|\geq 2}{\prod \Bl((S^3)^S, \Delta_S)} \x 
(S^3)^r \x \underset{S\subset \{q+1,...,q+r, q+r+1\}, |S|\geq 2}{\prod \Bl((S^3)^S, \Delta_S)} \ar[r] \ar@{^(->}[d] & 
(S^3)^{q+r} \x \underset{S\subset \{1,...,q+r+1\}, |S|\geq 2}{\prod \Bl((S^3)^S, \Delta_S)}{S} \ar@{^(->}[d] \\  
(S^3)^q \x (S^3)^r \x \R^{\langle L_1 \rangle , N_1} \x \R^{\langle L_2 \rangle , N_2} \ar[r] &
(S^3)^{q+r} \x \R^{\langle L_3 \rangle , N_3} }
\]
\end{sublemma}
\begin{proof}
For the top square, define the left-hand vertical map as follows.  Fix a diffeomorphism $h:\R\x D^2 \to \R^3$. Define the embedding $C_q(\R\x D^2) \incl (S^3)^q \x \prod_S \Bl((S^3)^S, \Delta_S)$ to be the obvious one after using $h\circ (\ell_1 \x id)$ to identify $\R\x D^2$ with a subspace of $S^3$.  Define the embedding of $C_r(\R\x D^2)$ similarly, but instead using $h \circ (\ell_2 \x id)$.  Then it is clear that if the map in the top row is the restriction of $\mu_C$ to the preimage of the interior of $C_{q+r}/\partial$ and the map in the middle row is the obvious one, then the square commutes.

We now justify the bottom square.  The maps on the factors of $S^3$ are obvious, while the ones involving the products of blowups require some attention.  For any $Q\in\bN$, embed 
\[
\underset{S\subset\{1,...,Q\}, |S|\geq 2}{\prod \Bl(M^S, \Delta_S)}
\]
into some Euclidean space with corners by neatly embedding each factor $\Bl(M^S,\Delta_S)$ into a Euclidean space with corners and then taking the product of these embeddings.  There is a partial order on such subsets $S$, where $S'\leq S$ if for any $q\in\N$ such that $S\subset \{1,...,q\}$, we have $S'\subset\{1,...,q\}$.  Fix a total order on finite subsets of $\bN$ compatible with this partial order, and arrange the product of these embeddings in this order.  We claim this suffices for the vertical maps in the bottom square.  In fact, we get an embedding of the right-hand product of blowups into some $\R^{\langle L_3 \rangle, N_3}$; the image of the restriction to the first product on the left-hand side is contained in some lower-dimensional $\R^{\langle L_1 \rangle, N_1}$; and we can set $L_2=L_3-L_1, N_2=N_3-N_1$, and the restriction of the second product on the left-hand side lands in the second block of coordinates, $\R^{\langle L_2 \rangle, N_2}$.  Our attention to the order on the subsets $S$ ensures that for any $Q$, the \emph{one} neat embedding of $C_Q(\R\x D^2)$ that we have constructed makes \emph{all} the relevant diagrams simultaneously commute.  (Strictly speaking, it is only the same embedding up to adding extra factors to the target Euclidean space with corners, but since we let this dimension get arbitrarily large, this is fine for our purposes.)
\end{proof}


We get a similar diagram for the $C_{q,t}$'s, shown below.  The top dotted arrow is the multiplication $C_{q+t}\x C_{r+s} \to C_{q+t+r+s}$, while the one to its right is induced by the permutation $\sigma$ which shifts the $r$ points to the left of the $t$ points, as defined in (\ref{sigmadefn}).  The map below it is induced by $\sigma^{-1}$.  We have abbreviated $\Bl_S:=\Bl((S^3)^S, \Delta_S)$.  The rest is as in the Sublemma.
\[
\xymatrix{
C_{q,t}\x C_{r,s} \ar@{-->}[r] \ar@{^(->}[d] & C_{q+t+r+s} \ar@{^(->}[d] \ar[r] & C_{q+r,t+s} \ar@{^(->}[d]\\
(S^3)^{q+t} \x {\prod \Bl_S} \x 
(S^3)^{r+s} \x {\prod \Bl_S} \ar[r] \ar@{^(->}[d] & 
(S^3)^{q+t+r+s} \x {\prod \Bl_S} \ar@{^(->}[d] & (S^3)^{q+t+r+s} \x {\prod \Bl_S} \ar[l]\\  
(S^3)^{q+t} \x (S^3)^{r+s} \x \R^{\langle L_1 \rangle , N_1} \x \R^{\langle L_2 \rangle , N_2} \ar[r] &
(S^3)^{q+t+r+s} \x \R^{\langle L_3 \rangle , N_3} & }
\]
The composition in the top row is the multiplication on the $C_{q,t}$'s, so the above simplifies to 
\[
\xymatrix{
C_{q,t} \x C_{r,s} \ar@{-->}[r] \ar@{^(->}[d] & C_{q+r,t+s} \ar@{^(->}[d] \\
(S^3)^{q+t} \x (S^3)^{r+s} \x \R^{\langle L_1 \rangle , N_1} \x \R^{\langle L_2 \rangle , N_2} \ar[r] &
(S^3)^{q+t+r+s} \x \R^{\langle L_3 \rangle , N_3} & }
\]

By embedding $S^3$ into Euclidean space, we finally get a square
\[
\xymatrix{
C_{q,t}(\R\x D^2)\x C_{r,s}(\R\x D^2) \ar@{-->}[r] \ar@{^(->}[d] & C_{q+r,t+s}(\R\x D^2) \ar@{^(->}[d] \\
\R^{\langle L' \rangle, N'} \x \R^{\langle L'' \rangle, N''} \ar[r] & \R^{\langle L \rangle, N}}
\]
which commutes, where $L' + L''=L, N' + N''=N$, where the dotted arrow is as before, and the map on the Euclidean spaces with corners is the obvious canonical one.  The above is precisely what we need for the bottom square in (\ref{embedonefactorinterior}) to commute; the top square in that diagram commutes just because the map in the top row is a restriction of the map in the middle row.  This completes the proof of Lemma \ref{embmultcommutesonefactor}. 
\end{proof} 


\bigskip
\begin{lemma}
\label{normalpullback}
In diagram (\ref{embedonefactorinterior}) the normal bundle to the left-hand composite embedding
\[U\subset E_{q,t} \x E_{r,s} \incl \K \x \bR^{\langle L'\rangle,N'} \x \K \x \bR^{\langle L'' \rangle, N''} 
\]
 is the pullback of the normal bundle to the right-hand composite embedding 
\[E_{q+r,t+s} \incl \K \x \R^{\L, N}.
\]
\end{lemma}
\begin{proof}
All three  horizontal maps in that diagram are actually embeddings, so the left-hand normal bundle is a subbundle of the pullback.  But it is easy to see that the dimensions of the left and right normal bundles are equal, hence the left-hand normal bundle is the whole pullback bundle.  (Alternatively, the preimage of a tubular neighborhood of the top right-hand embedding is a tubular neighborhood of the top left-hand embedding, because either one is the space of $(f_1,f_2,x_1,...,x_{q+t+r+s})$ with $q+r$ of the $x_i$ on the knot $f_1 \# f_2$; for the bottom square, clearly the left-hand normal bundle is the pullback of the right-hand normal bundle.)  
\end{proof}

This fact ensures that the diagram obtained after taking Thom collapse maps commutes:

\begin{lemma}
\label{collapsecommutes}
Taking Thom collapse maps and quotienting by boundaries in (\ref{embedonefactorinterior}) gives the following commutative diagram, where $M=L+N$ and similarly for $M', M''$, and where each superscript $\nu$ indicates the Thom space of the normal bundle of the appropriate embedding:
\begin{equation}
\xymatrix{
E^\nu_{q,t}/\partial \wedge E^\nu_{r,s}/\partial \ar[r] & U^\nu/\partial \ar[r] & E^\nu_{q+r,t+s} \\
\K_+\wedge S^{M'} \wedge \K_+\wedge S^{M''} \ar[u]\ar@{=}[r] & \K_+\wedge S^{M'} \wedge \K_+\wedge S^{M''} \ar[u]\ar[r] & \K_+\wedge S^M \ar[u]} \label{multthomspaces}
\end{equation}
\end{lemma}
\begin{proof}
Lemma \ref{normalpullback} implies that the right-hand square commutes.  Since $U$ is contained in the interior of $E_{q,t} \x E_{r,s}$, the top left horizontal map is just a quotient map, and the left-hand square commutes because the embedding of $U$ into Euclidean space was just the restriction of the embedding of $E_{q,t} \x E_{r,s}$.
\end{proof}



\subsection{Multiplication and the Thom spectrum}
\label{thomringspectrumsection}

At this point, we will both (a) apply disjoint union to previous diagrams to get one space with a multiplication and (b) let the dimension of the Euclidean space increase to get a spectrum with a multiplication, i.e., a ring spectrum.  This will complete the proof of Theorem \ref{mapofringspectra}; the previous section already provided many of the necessary ingredients.  We recall the statement here for the reader's convenience.  Here and below, by a ring spectrum $R$ we mean a ring spectrum in the homotopy sense.  That is, $R$ has a multiplication $R\wedge R \to R$ which is associative only up to homotopy, and it has a map $S^0 \to R$ which is a unit for the multiplication only up to homotopy.

\begin{T1}
The total spaces $E_{q,t}$ of the bundles over the knot space have a multiplication $E_{q,t}/\d E_{q,t}\x E_{r,s}/\d E_{r,s}\to E_{q+r,t+s}/\d E_{q+r,t+s}$ which makes a certain wedge of Thom spectra $\bigvee_{q,t \in \bN} E^\nu_{q,t}/\d E^\nu_{q,t}$ into a ring spectrum.  The Thom collapse maps for various $q,t$ induce a map of ring spectra
\[
\bigvee_{q,t} \Sigma^\infty \K_+ \to \bigvee_{q,t} E^\nu_{q,t}/\d E^\nu_{q,t}
\]
where the multiplication in $\bigvee \Sigma^\infty \K_+$ comes from the space-level connect-sum.  This induces an ``integration along the fiber'' map in cohomology with arbitrary coefficients, producing classes in $H^*\K$.
\end{T1}


\begin{proof}[Proof of Theorem \ref{mapofringspectra}]
We define the Thom spectrum $\bigvee \E$ precisely as follows.  The fact that only finitely many factors in the disjoint union embed into $\R^{\L, N}$ for fixed $L,N$ leads to slightly more complicated indexing.  Let $M\in \bN$, and let $L(M),N(M)$ be nondecreasing sequences of natural numbers so that $L(M)+N(M)=M$ and $L(M),N(M)\to\infty$ as $M\to\infty$.  Then define $(\bigvee \E)_M:=\bigvee_{q,t\in I(M)} E^\nu_{q,t}$, where $I(M)$ is the index set of pairs of natural numbers $q,t$ such that $C_{q+t}$ embeds into $\bR^{\langle L(M)\rangle, N(M)}$, and where $\nu=\nu_{q,t,M}$ is the normal bundle to the embedding $E_{q,t} \incl \K \x \bR^{\langle L(M) \rangle, N(M)}$.  The structure maps are just wedges of the structure maps of $\E$.  There is a subspectrum $\partial \bigvee E^\nu$, whose $M^{\textrm{th}}$ space is just the wedge of the boundary subspaces of the appropriate Thom spaces, and we consider the quotient $\bigvee E^\nu/\partial$ by this subspectrum.  Since our multiplication is only defined on the quotients by boundary, we can make only $\bigvee E^\nu/\partial$, and not $\bigvee E^\nu$ itself, into a ring spectrum.   Of course, $\mu: \K\x\K\to \K$ makes $\Sigma^\infty \K$ into a ring spectrum.

We first take a disjoint union of diagrams (\ref{embedonefactorinterior}) which commute by Lemma \ref{embmultcommutesonefactor}.  For any $M'+M''=M$, that gives us the square below, where $L=L(M), L'=L(M')$, etc.  Note that the conditions on $L,N$ imply that $L'+L''=L$, $N'+N''=N$. To get a map where the dashed arrow is, we must not only restrict to open subsets of the factors as in (\ref{embedonefactorinterior}), but also to the factors indexed by $q,t,r,s$ such that $E_{q+r,t+s}$ can be embedded into $\bR^{\L,N}$.
\begin{equation}
\xymatrix{
\coprod_{q,t\in I(M')} E_{q,t} \x \coprod_{q,t\in I(M'')} E_{q,t} \ar@{-->}[r] \ar@{^(->}[d] & \coprod_{q,t\in I(M)}E_{q,t}\ar@{^(->}[d] \\
\K \x \coprod_{q,t\in I(M')} \bR^{\langle L' \rangle, N'} \x \K \x \coprod_{q,t\in I(M'')} \bR^{\langle L'' \rangle, N''}
\ar[r]& \K \x \coprod_{q,t\in I(M)} \bR^{\L, N} }\label{embeddisjointunion}
\end{equation}

The normal bundle to the embedding of each disjoint union is of course the disjoint union of the normal bundles of the factors; the normal bundle to the embedding of the product on the left is the product of the normal bundles to the embeddings of the factors. Thus the Thom space of the normal bundle to the upper-left space is a smash product of two wedge products of Thom spaces.  So taking the Thom collapse of the vertical embeddings and compactifying the ambient spaces induces the top square below, which is just a wedge (or disjoint union) of diagrams (\ref{multthomspaces}), which commute by Lemma \ref{collapsecommutes}.  The top arrow is still only defined on certain factors $E_{q,t}/\partial \wedge E_{r,s}/\partial$, but quotienting by boundaries allows it to be defined on all of such a factor $E_{q,t}/\partial \wedge E_{r,s}/\partial$.  Quotienting by the whole boundaries also causes the normal bundles $\nu$ not to depend the $L(M)$'s or $N(M)$'s, but only the $M$'s.  The bottom square, which will be useful shortly, just comes from an inclusion induced by a choice of $q,t$ in the first wedge sum and $r,s$ in the second wedge sum.

\begin{equation}
\xymatrix{
\bigvee_{q,t\in I(M')} E^{\nu_{M'}}_{q,t}/\partial \wedge \bigvee_{q,t\in I(M'')} E^{\nu_{M''}}_{q,t}/\d \ar@{-->}[r] & \bigvee_{q,t\in I(M)}E_{q,t}^{\nu_M} /\d \\
\K_+ \wedge \bigvee_{q,t\in I(M')} S^{M'}\wedge \K_+ \wedge \bigvee_{q,t\in I(M'')} S^{M''} \ar[r] \ar[u] & \K_+ \wedge \bigvee_{q,t\in I(M)} S^M \ar[u] \\
\K_+ \wedge S^{M'} \wedge \K_+ \wedge S^{M''} \ar[r] \ar[u]^{\tau_{q,t}\wedge \tau_{r,s}} & \K_+ \wedge S^M \ar[u]_{\tau_{q+r,\> t+s}}}\label{thomdisjointunion}
\end{equation}

We will now see that $\bigvee \E/\partial$ is a ring spectrum.  To get a well defined spectrum $\bigvee \E \wedge \bigvee \E$ in the homotopy category, we can set 
\[\left(\bigvee \E/\partial \wedge \bigvee \E/\partial\right)_M=\bigvee \E/\partial_{M'} \wedge \bigvee \E/\partial_{M''}
\]
for any choice of $M',M''$ satisfying $M'+M''=M$, as long as $M',M''\to \infty$ as $M\to \infty$ (cf. Adams' book \cite{Adams}).  Notice also that the subspectrum on which the dotted arrow is defined is cofinal, which suffices for a map (again cf. \cite{Adams}) of spectra
\[\bigvee \E/\partial \wedge \bigvee \E/\partial \to \bigvee \E/\partial,
\]
so $\bigvee \E/\partial$ is a ring spectrum.  Notice that the cofinality of the sequences $M', M'',L$, and $N$ ensures that the ring spectrum structure encodes multiplication of a configuration of $q+t$ points by one of $r+s$ points for any $q,t,r,s$.

Diagram \ref{thomdisjointunion} implies that we have a map of ring spectra 
\[
\tau: \bigvee \Sigma^\infty \K_+  \to \bigvee E^\nu \label{ringspectra}
\]
where $\left( \bigvee \Sigma^\infty \K_+ \right)_M := \K_+ \wedge \bigvee_{q,t \in I(M)} S^M$ has a multiplication induced by connect-sum (i.e., it is a wedge of the multiplication maps for the suspension spectrum $\Sigma^\infty \K_+$).


All that remains is to exhibit maps in cohomology whose target is $H^*\K$.  To achieve this, we just restrict to one factor at a time in the wedge on the left-hand side.  We have maps as in (\ref{thomdisjointunion})
\[
\tau_{q,t}:\Sigma^\infty \K_+ \to \bigvee E^\nu \label{quotientringspectra}.
\]
Via the suspension isomorphisms, the image of $\tau_{q,t}^*$ is in $H^*\K$.
\end{proof}

We note that having a map of ring spectra includes the relation
\[
\mu_E \circ (\tau_{q,t} \wedge \tau_{r,s}) \simeq \tau_{q+r,t+s} \circ \mu,
\]
where the two sides don't agree on the nose because we are using an up-to-homotopy notion of ring spectra.  Still, this means that induced maps in homology agree, which will be enough for our purposes. 

\begin{remark}
Restricting to boundary strata of the $C_{q+t}$'s corresponds to restricting to boundary strata in $E_{q,t}$ (by taking the preimage of the embedding $E_{q,t}\incl \K \x C_{q+t}$).  The multiplication on the $E_{q,t}/\d$'s can be defined on strata of the $E_{q,t}$'s modulo their boundaries just as in the case of the configuration spaces, as explained in remark \ref{rmkmultfaces}.  So while in diagrams (\ref{multthomspaces}) and (\ref{thomdisjointunion}) we brutally quotient by the whole boundaries, we have a diagram like (\ref{multthomspaces}) for each pair of strata in $E_{q,t}$ and $E_{r,s}$.   The Thom spaces in the top row would be replaced by the subspaces which are Thom spaces of these strata, while the bottom row would remain the same, since the restrictions of the neat embeddings of the whole spaces are neat embeddings of the strata. 
\qed
\end{remark}

We finally note that we did not rely on $\R^3$ for the proof of this theorem, since we did not make full use of the little 2-cubes action on $\K$.  It is easy to see that a space-level connect-sum can be defined for long knots in $\R^n$ for any $n \geq 3$, so the Theorem extends to this case.  We leave it to the reader to verify the details.

\subsection{Computing the coproduct on cohomology of configuration spaces}
\label{multconfhomology}

We have established the compatibility of the Thom collapse maps with the multiplications on the spaces involved, so we would now like to explicitly compute the coproduct $\mu_E^*$ on certain classes in $H^*(\coprod E_{q,t}/\partial)$.  In order to do this, we will need to understand the coproduct on $H^*(\coprod (C_{q,t}/\d))$, but we start by considering the closely related map on $H^*(\coprod C_q/\partial)$.  To understand it, we consider a space-level coproduct on $\coprod C_q$. 

\begin{definition}
Define a space-level coproduct $\delta$ on $\coprod C_q$ as the map whose components $\delta_{pq}: C_{p+q}\to C_p\x C_q$ are given as follows.  There is a map 
\[
(S^3)^{p+q} \x \prod_{\substack{S\subset \{1,...,p+q\} \\ |S|\geq 2}} \Bl((S^3)^S, \Delta_S) \to
(S^3)^p \x \prod_{\substack{S\subset \{1,...,p\} \\ |S|\geq 2}} \Bl((S^3)^S, \Delta_S) \x 
(S^3)^q \x \prod_{\substack{S\subset \{p+1,...,p+q\} \\ |S|\geq 2}} \Bl((S^3)^S, \Delta_S)
\]
induced by the obvious diffeomorphism $(S^3)^{p+q}\to (S^3)^p\x (S^3)^q$ and the obvious projection on the products of blowups.  Take $\delta_{pq}$ to be the restriction of this map to the compactified configuration spaces. 
\end{definition}

This map does not take boundary to boundary, but it is an embedding when restricted to the interior of its domain.  We will show 

\begin{proposition}
\label{mudualtodelta}
The multiplication $\mu_C$ is homotopic to the Thom collapse map of this embedding $\delta$ and hence is Spanier-Whitehead dual to it.
\end{proposition}

Before starting the proof, we briefly review Spanier-Whitehead duality.  In general, this duality gives spectra dual to spaces, unlike Poincar\'{e} duality, which is just at the level of homology.  (We will end up with a map of spaces instead of spectra, but it is still natural to proceed via Spanier-Whitehead duality.)  The Spanier-Whitehead dual of a space $X$ is a spectrum $\map(X,S)$ whose $n^{\textrm{th}}$ space is the space of maps from $X$ to $S^n$.  A map of spaces $f:X\to Y$ induces a map of Spanier-Whitehead duals $f_!:\map(Y,S) \to \map(X,S)$.  In the case that $f$ is an embedding of manifolds, a theorem of Atiyah \cite{Atiyah} says that the Spanier-Whitehead dual $\map(X,S)$ is the Thom spectrum $X^\nu$ of the normal bundle of $X$ and that the map $f_!$ is given by the Thom collapse map. When $X$ has boundary, the dual is the quotient $X^\nu/ \partial X^\nu$ of Thom spectra.

\begin{proof}[Proof of Proposition \ref{mudualtodelta}]
All of the above remarks apply to $\delta$ because it is an embedding when restricted to the interior.  Since the ``open" configuration spaces are open subsets of Euclidean space, they are framed manifolds (i.e., their normal bundles are trivial), so the dual to $\delta$ is a map of suspension spectra
\[
\delta_! : \Sigma^\infty (C_p \x C_q)_+ / \Sigma^\infty \partial(C_p \x C_q)_+ \to \Sigma^\infty (C_{p+q})_+ / \Sigma^\infty \partial(C_{p+q})_+.
\]
Remember that for any choice of nonsymmetric little intervals object $(\ell_1, \ell_2)$, $\mu_C$ is homotopic to a multiplication given by a trivial action of the little intervals.  On the interior of $C_{p+q}$, the composition of $\delta$ followed by that multiplication is the identity.  Thus $\mu$ is homotopic to the Thom collapse map on the image of the interior, and furthermore either map collapses everything else to the basepoint.  So $\mu_C: C_p/\partial \wedge C_q/\partial \to C_{p+q}/\partial$ induces the map of spectra $\delta_!$ and is hence dual to $\delta$ in (co)homology.
\end{proof}


\begin{remark}\label{mucirc} Similarly, instead of $\mu_C$ we could consider a map $\overset{\circ}{\mu}_C:\inte(C_p\x C_q)\to \inte \>C_{p+q}$ defined as follows.  Identifying the interior of $I\x D^2$ with $\bR^3$, we see that there is an obvious nonsymmetric little intervals action on $\inte \coprod C_q$ given by putting configurations of points next to each other in the long direction; this induces the multiplication.  Notice that it is an embedding.  The map $\delta$ induces $\bar{\delta}: C_{p+q}/\partial \to C_p/\partial \wedge C_q/\partial$, and the composition $\bar{\delta} \circ \overset{\circ}{\mu}_C$ is homotopic to the identity.  So we see that $\bar{\delta}$ is (homotopic to) the Thom collapse map dual to the embedding $\overset{\circ}{\mu}_C$.  For our purposes $\overset{\circ}{\mu}_C$ is not so useful because we will need a multiplication on quotients, and $\overset{\circ}{\mu}_C$ fails to take boundary to boundary.
\qed
\end{remark}

Next recall from \ref{BTsection} that we have classes $\omega_{ij}$ for $1\leq i < j \leq q$ which generate the cohomology of $C_q$ as an algebra.
These classes and their duals $\omega_{ij}^*$ in homology can be thought of as ``spherical" classes since geometrically, they correspond to the $i^{\textrm{th}}$ point moving in a sphere around the $j^{\textrm{th}}$ point.  We now compute the product induced by $\delta$ on cohomology\footnote{More precisely, we have $\xymatrix{ H^*(\coprod C_q) \otimes H^*(\coprod C_q)\ar[r]^-{\x} & H^*(\coprod C_q \x \coprod C_q) \ar[r] & H^*(\coprod C_q)}$, where the cross-product $\x$ is injective, though not an isomorphism because there are infinitely many components.  When we consider the dual coproduct, we will only consider elements in $\bigoplus_q H^*C_q < \prod_q H^*C_q \cong H^*(\coprod_q C_q)$, which are mapped by this coproduct into the image of the cross-product, on which the cross-product has a well defined inverse.  Thus this issue does not arise when we confuse the induced map in (co)homology with the (co)product obtained by composing with the cross-product.}:

\begin{lemma}\label{proddeltaqr}
The components $\delta_{qr}:C_{q+r}\to C_q \x C_r$ of $\delta$ satisfy the formulae below
\begin{align*}
\delta_{qr}^*(\omega_{ij} \otimes 1)&= \omega_{ij} \\
\delta_{qr}^*(1 \otimes \omega_{ij})&= \omega_{q+i,q+j}
\end{align*}
where the different $\omega_{ij}$'s and $1$'s are elements of different algebras; to which one each belongs is clear from the domain and range of $\delta_{qr}$.
Hence 
\[
\delta_{qr}^*(\omega_{i_1 j_1}\cup...\cup \omega_{i_m j_m}\> \otimes \> \omega_{k_1 l_1}\cup...\cup \omega_{k_n l_n}) = 
\omega_{i_1 j_1} \cup... \cup \omega_{i_m j_m} \cup \omega_{q+k_1, q+l_1} \cup...\cup \omega_{q+k_n, q+l_n},
\]
and since any class in $H^*C_q$ is a sum of monomials in the $\omega_{ij}$, this completely describes the product.
\end{lemma}

\begin{proof}
We can write each $\delta_{qr}=\delta_q \x \delta_r$ for maps $\delta_q:C_{q+r}\to C_q$ and $\delta_r: C_{q+r}\to C_r$ which project to the first $q$ and last $r$ points, respectively.  We can directly see that
\[
\xymatrix{
(\omega_{ij}\in H^*C_q) \ar@{|->}[r]^-{\delta_q^*} & (\omega_{ij}\in H^*C_{q+r}) \\
(\omega_{ij}\in H^*C_r) \ar@{|->}[r]^-{\delta_r^*} & (\omega_{q+i,q+j}\in H^*C_{q+r}) }
\]
By basic properties of the cross-product, $\delta_{qr}^*(\theta \x \eta) =  \delta_q^* \theta \cup \delta_r^*\eta$ for any $\theta, \eta$, so
\begin{align*}
\delta_{qr}^*(\omega_{ij} \x 1)&= \omega_{ij} \\
\delta_{qr}^*(1 \x \omega_{ij})&= \omega_{q+i,q+j}
\end{align*}
If we abuse notation and let $\delta_{qr}^*$ denote instead the cross-product followed by the induced map in cohomology, we can replace the $\x$'s by $\otimes$'s; in general we would have to introduce a factor of $\pm 1$ coming from the relation between cup- and cross-product, but since all our classes are even-dimensional this sign is always $+1$.  Another consequence of this is that $\delta_{qr}^*$ takes the product in $H^*C_q \otimes H^* C_r$ given by $\cup \otimes \cup$ to cup-product in $H^*C_{q+r}$.  Thus
\begin{align*}
\delta_{qr}^*(\omega_{i_1 j_1}\cup...\cup \omega_{i_m j_m}\> \otimes \> \omega_{k_1 l_1}\cup...\cup \omega_{k_n l_n}) &= \\
\delta_{qr}^*((\omega_{i_1 j_1}\otimes 1)\cup...\cup (\omega_{i_m j_m}\otimes 1)\> \cup \> (1\otimes \omega_{k_1 l_1})\cup...\cup (1\otimes \omega_{k_n l_n})) &= \\
\delta_{qr}^*(\omega_{i_1 j_1}\otimes 1)\cup...\cup \delta_{qr}^*(\omega_{i_m j_m}\otimes 1)\> \cup \> \delta_{qr}(1\otimes \omega_{k_1 l_1})\cup...\cup \delta_{qr}^*(1\otimes \omega_{k_n l_n})) &= \\
\omega_{i_1 j_1} \cup... \cup \omega_{i_m j_m} \cup \omega_{q+k_1, q+l_1} \cup...\cup \omega_{q+k_n, q+l_n}
\end{align*}
which proves the last formula.
\end{proof}

Combining Lemma \ref{mudualtodelta} and Lemma \ref{proddeltaqr} immediately gives 


\begin{proposition} 
\label{coprodCq} 

The space $\coprod_q (C_q/\d)$ has a multiplication $\mu_C$ making it into a (disconnected) $H$-space.  The induced coproduct $\mu_C^*$ on cohomology is dual to the product $\delta^*$ on $H^*\left(\coprod_q C_q\right)$, as described in Lemma \ref{proddeltaqr}.  Hence the product $(\mu_C)_*$ on $H_*(\coprod_q (C_q/\d))$, as dual to $\mu_C^*$, corresponds to the product $\delta^*$.  \qed


\end{proposition}

%



\subsection{Product formula for evaluations of Bott-Taubes classes on connect-sums of knots}
\label{productformulasection}

Having computed the product on $\coprod_q (C_q/\d)$ in (co)homology, we still need to compute the product on $\coprod_{q,t} (C_{q,t}/\d)$ as the last step before deducing the promised product formula.  The following lemma is essentially an addendum to the previous section.  As one might expect, it just involves permuting certain indices.

\begin{lemma}\label{coprodCqt} 
The coproduct in cohomology induced by the multiplication $\mu_C: (C_{q,t}/\d) \x (C_{r,s}/\d) \to C_{q+t,r+s}/\d$ is dual to the product given by
\begin{align*}
\omega_{ij} \otimes 1 &\mapsto \omega_{\sigma(i) \sigma(j)} \\
1\otimes \omega_{ij} &\mapsto \omega_{\sigma(i+(q+t)) \sigma(j+(q+t))} 
\end{align*}
where $\sigma$ is the permutation of $\{1,...,q+t+r+s\}$ which shifts the $r$ letters to the left of the $t$ letters, as defined by (\ref{sigmadefn}).  Hence the product in homology induced by $\mu_C$ corresponds to the product defined above. 
\end{lemma}

\begin{proof}
Recall that this product on the space level is the composition 
\[
\xymatrix{
C_{q+t}/\d \x C_{r+s}/\d \ar[r]^-{\mu_C} & C_{q+t+r+s}/\d \ar[r]^-\sigma & C_{q+t+r+s}/\d
}
\]
where the second map is the diffeomorphism induced by $\sigma$.  Thus it is Spanier-Whitehead dual to the composition
\[
\xymatrix{
C_{q+t+r+s} \ar[r]^{\sigma^{-1}} & C_{q+t+r+s} \ar[r]^-\delta &  C_{q+t} \x C_{r+s}
}
\]
where the first map is the diffeomorphism induced by the inverse permutation.  But we know $\delta^*$ from Proposition (\ref{coprodCq}), and in cohomology, the first map clearly sends $\omega_{ij}$ to $\omega_{\sigma(i)\sigma(j)}$.  This completely determines this dual product and hence our coproduct in cohomology.  The product in homology is of course dual to this coproduct and hence corresponds to the product $(\delta \circ \sigma^{-1})_*$.
\end{proof}

We are now ready to derive the product formula, which we restate for convenience:

\begin{T2}
Let $\beta \in H^*(C_q(\R^3)/\d C_q(\R^3))$ and $a_1,a_2 \in H_*\K$.  Let $\theta_i$ and $\eta_i$ be classes in $H^*(\coprod C_q(\R^3) /\d C_q(\R^3))$ such that $\mu_C^* \beta =\sum_i \theta_i \otimes \eta_i$.  Then
\[
\langle \tau^* ev^* \beta, \> \mu_*(a_1 \otimes a_2)\rangle = \sum_i \langle \tau^* ev^* \theta_i,\> a_1\rangle \cdot 
\langle \tau^* ev^* \eta_i, a_2 \rangle,
\]
where the cohomology can be taken with coefficients in any ring.
\end{T2}

\begin{proof}[Proof of Theorem \ref{productformula}, the product formula]
Recall that $ev$ induces a map $\bigvee_{q,t} (E_{q,t}/\d) \to \bigvee_{q,t} (C_{q,t}/\d)$.  Given $\beta\in H^*(C_q/\d)$, we want to evaluate  $\tau^* ev^* \beta$ on a homology class in $H_*\K$ coming from a connect-sum of knots (where we have omitted the Thom isomorphism and suspension isomorphism from the notation).  

We claim the diagram below commutes:
\begin{equation}
\xymatrix{
\tH^*(C_{q,t} \x C_{r,s}/\d)  \ar[d]_{(ev\x ev)^*} &
\tH^*(C_{q+t,r+s}/\d) \ar[d]^{ev^*} \ar[l]_-{\mu_C^*}\\
\tH^*(E_{q,t} \x E_{r,s}/\d) \ar[d]_{\mbox{Thom }\cong} &
\tH^*(E_{q+r,t+s}/\d) \ar[d]^{\mbox{Thom }\cong} \ar[l]_-{\mu_E^*} \\
\tH^*(E_{q,t}^\nu \wedge E_{r,s}^\nu/ \partial) \ar[d]_{(\tau \wedge \tau)^*}&
\tH^*(E_{q+r,t+s}^\nu/ \partial) \ar[l]_-{\mu_E^*} \ar[d]^{\tau^*}\\
\tH^*(\Sigma^\infty\K_+ \wedge \Sigma^\infty \K_+) & H^*(\Sigma^\infty \K_+) \ar[l]_-{\mu^*} }\label{first}
\end{equation}
The top square commutes because we defined the multiplications $\mu$ so that the maps commute on the space-level; the middle square does too, by naturality of the Thom isomorphism; and the bottom square does since we established that the multiplication commutes up to homotopy with the Thom collapse map.  (We also have commutativity of $\mu^*$ with the suspension isomorphisms, though we have omitted this last square from the diagram.) 

Let $a_1, a_2 \in H_*\K \cong H_*(\Sigma^\infty \K_+)$.  Below we let $\mu_*$ and $\mu^*$ denote compositions of the induced maps in (co)homology with the cross product or its inverse (cf. previous footnote).  Then
\begin{align*}
\langle \tau^* ev^* \beta, \> \mu_*(a_1\otimes a_2) \rangle &= 
\langle \mu^* \tau^* ev^* \beta, \>a_1 \otimes a_2\rangle \mbox{ by duality of coproduct $\mu^*$ to product $\mu_*$} \\
&= \langle \tau^* ev^* \mu_C^*\beta, \> a_1 \otimes a_2 \rangle \mbox{ by (\ref{first})}
\end{align*}
By Proposition (\ref{coprodCq}), given $\beta$ as a sum of products of $\omega_{ij}$'s, we can write $\mu_C^*\beta$ as a sum of tensors of cohomology classes, say $\sum_i \theta_i \otimes \eta_i$.  Then we have
\[
\langle \tau^* ev^* \beta,\> \mu_*(a_1\otimes a_2) \rangle = 
\sum_i \langle \tau^* ev^* \theta_i,\> a_1 \rangle \cdot \langle \tau^* ev^* \eta_i,\> a_2 \rangle,
\]
as desired.
\end{proof}

We emphasize that this formula holds for any coefficient ring.  Also, as with Theorem \ref{mapofringspectra}, Theorem \ref{productformula} is still true if $\K$ is replaced by $\Emb(\R,\R^n)$ for $n\geq 3$.  Roughly, the formula says that given a cohomology class coming from a configuration space, we can evaluate it on a connect-sum of knots by decomposing it into products of classes coming from configurations of fewer points, and then summing the products of the evaluations of those smaller classes on the respective knots.

\begin{example}
Let $\beta\in H^4(C_4/\d)$ be the class Poincar\'{e} dual to $(\omega_{12}\>\omega_{34})_4 \in C_4$ where the subscript 4 outside the parentheses indicates that this is an element of $C_4$.  Using similar notation, the elements which map to $(\omega_{12}\>\omega_{34})_4$ under $\delta^*$ are $1_0 \otimes (\omega_{12}\>\omega_{34})_4$, $(\omega_{12})_2 \otimes (\omega_{12})_2$, and $(\omega_{12}\>\omega_{34})_4\otimes 1_0$.  If we let $\overline{\alpha}$ denote the Poincar\'{e} dual of a class $\alpha$, the product formula gives 
\begin{align*}
\langle \tau^* ev^* \overline{(\omega_{12}\>\omega_{34})_4},\> \mu_*(a_1 \otimes a_2) \rangle =& 
\langle \tau^* ev^* \overline{1_0},\> a_1\rangle \cdot \langle \tau^* ev^* \overline{(\omega_{12}\>\omega_{34})_4 },\> a_2 \rangle + \\
& \langle \tau^* ev^* \overline{(\omega_{12})_2},\> a_1\rangle \cdot \langle \tau^* ev^* \overline{(\omega_{12})_2 },\> a_2 \rangle
+ \\
& \langle \tau^* ev^* \overline{(\omega_{12}\>\omega_{34})_4},\> a_1\rangle \cdot \langle \tau^* ev^* \overline{1_0},\> a_2 \rangle.
\end{align*}
\qed
\end{example}

\subsection{Bracket operation (or lack thereof)}
\label{bracket}
We might next try to find a similar formula for 
\[\langle \tau^* ev^* \beta,\> \xi_*([S^1]\otimes a_1 \otimes a_2)\rangle
\]
where 
\[
\xi: \C_2(2) \x \K \x \K \to \K
\]
is the structure map for the little 2-cubes action, and of course $\C_2(2)\simeq S^1$.  (With $\bQ$ coefficients $\xi_*([S^1]\otimes a_1 \otimes a_2)$ is the Browder operation, often denoted $[a_1, a_2]$ or $\lambda_1(a_1, a_2)$, and with $\bZ/2$ coefficients, $\xi_*([S^1]\otimes K \otimes K)$ is $Q_1(K)$, one of the Dyer-Lashof operations.)  The first obstruction to doing so is the following proposition concerning the failure of the multiplication on the total space of our bundle.  The main point is that we are considering \emph{ordered} configurations of points, which results in noncommutativity of the multiplications in homology.

\begin{proposition}\label{no2cubeslift}
Let $\mu_C$ be the multiplication on $\coprod_q (C_q/\d)$ as previously defined.  Suppose we have a multiplication $\mu_E$ on $\coprod_{q,t} (E_{q,t}/\d)$ compatible with $ev: E_{q,t}/\d \to C_{q,t}/\d$, i.e., 
\[
 \xymatrix{
\coprod_{q,t} (C_{q,t}/\d) \x \coprod_{q,t} (C_{q,t}/\d) \ar[r]^-{\mu_C} & \coprod_{q,t} (C_{q,t}/\d)  \\
\coprod_{q,t} (E_{q,t}/\d) \x \coprod_{q,t} (E_{q,t}/\d) \ar[r]^-{\mu_E} \ar[u]^{ev \x ev} & \coprod_{q,t} (E_{q,t}/\d) \ar[u]_{ev}
}
\]
commutes.  Then $\mu_E$ does not extend to a little 2-cubes action on $\coprod_{q,t} (E_{q,t}/\d)$.
\end{proposition}

\begin{proof}
It suffices to show that $\mu_E$ is not commutative in homology.  The factors $E_{0,t}$ are trivial bundles over the knot space, and in that case $ev$ is just the projection onto the fiber, which is surjective in homology.  Now $(\mu_C)_*$ is not commutative, as can be seen immediately from Lemma \ref{coprodCqt}.  Combining these facts with the commutativity of the square above completes the proof.
\end{proof}

This Proposition relies on the multiplication we have defined on the quotients of configuration spaces.  We could still ask whether there is \emph{another} multiplication on configuration space and one on our total space (both compatible with each other and with connect-sum) which would allow for a calculation analogous to that in Theorem \ref{productformula}.  The following Proposition shows that \emph{any} compatibly defined multiplications on these spaces cannot give interesting bracket operations in homology.  It holds essentially for dimensional reasons.

\begin{proposition}
\label{bracketprop}
Suppose one is given any little 2-cubes actions on $(\coprod E_{q,t})/\partial$ and $(\coprod C_{q+t}(\R\x D^2))/\partial$ compatible with Budney's little 2-cubes action on $\K$ and the maps $\xymatrix{E_{q,t}/\partial \ar[r]^-{ev} & C_{q+t}(\R\x D^2).\partial}$ and $\xymatrix{ \Sigma^\infty \K_+ \ar[r]^-\tau & E_{q,t}^\nu}$.  Then  
\[\langle \tau^* ev^* \beta,\> \xi_*([S^1]\otimes a_1 \otimes a_2)\rangle=0
\]
where $\xi$ is the structure map for the little 2-cubes action on $\K$.
\end{proposition}
\begin{proof}
Suppose we have a diagram 
\begin{equation}
\xymatrix{
H^*(S^1)\otimes \tH^*(C_{q+t}\x C_{r+s}/\d)  \ar[d]_{\id \otimes (ev\x ev)^*} &
\tH^*(C_{q+t+r+s}/\d) \ar[d]^{ev^*} \ar[l]_-{\xi^*}\\
H^*(S^1) \otimes \tH^*(E_{q,t} \x E_{r,s}/\d) \ar[d]_{\id \>\otimes \mbox{ Thom }\cong} &
\tH^*(E_{q+r,t+s}/\d) \ar[d]^{\mbox{Thom }\cong} \ar[l]_-{\xi^*} \\
H^*(S^1) \otimes \tH^*(E_{q,t}^\nu/\partial) \otimes \tH^*(E_{r,s}^\nu/ \partial) \ar[d]_{\id \otimes (\tau \wedge \tau)^*}&
\tH^*(E_{q+r,t+s}^\nu/ \partial) \ar[l]_-{\xi^*} \ar[d]^{\tau^*}\\
H^*(S^1) \otimes \tH^*(\Sigma^\infty\K_+ \wedge \Sigma^\infty \K_+) & H^*(\Sigma^\infty \K_+) \ar[l]_-{\xi^*} }\label{second}
\end{equation}
Then 
\begin{equation}
\langle \tau^* ev^* \beta,\> \xi_*([S^1]\otimes a_1 \otimes a_2)\rangle =
\langle \xi^*\tau^* ev^* \beta,\> [S^1]\otimes a_1 \otimes a_2 \rangle =
 \langle \tau^* ev^* \xi^*\beta,\> [S^1]\otimes a_1 \otimes a_2 \rangle \label{bracketeqn}
\end{equation}
 But by the results of \cite{Arnold, FCohen} we know that the (co)homology of each $C_q$ is nonzero only in even dimensions.  Poincar\'{e} duality implies that the (co)homology of $C_q/\partial$ 
is nonzero only in odd or even dimensions, according as $q$ is odd or even.  It is then easily seen that the factor $\xi^*\beta$ in $H^*(S^1)$ must be entirely contained in $H^0(S^1)$, so the expression in the right-hand side of (\ref{bracketeqn}) must be zero.  We emphasize that this argument made no assumptions about how we lifted the space-level connect-sum to the total space.
\end{proof}





\bibliographystyle{plain}
\bibliography{refs}

\begin{thebibliography}{10}

\bibitem{Adams}
J.F. Adams.
\newblock {\em Stable homotopy and generalized homology}.
\newblock Univ. of Chicago Press, 1972.

\bibitem{Arnold}
V.I. Arnold.
\newblock The cohomology ring of the colored braid group.
\newblock {\em Mat. Zametki}, 5(2):227--231, 1969.

\bibitem{Atiyah}
M.~Atiyah.
\newblock Thom complexes.
\newblock {\em Proc. London Math. Soc.}, 3(1):291--310, 1961.

\bibitem{Axelrod-Singer}
S.~Axelrod and I.M. Singer.
\newblock {Chern-Simons perturbation theory II}.
\newblock {\em J. Diff. Geom.}, 39(1):173--213, 1994.

\bibitem{BNthesis}
D.~Bar-Natan.
\newblock {\em {Perturbative aspects of Chern-Simons topological quantum field
  theory}}.
\newblock PhD thesis, Princeton University, 1991.

\bibitem{Becker-Gottlieb}
J.C. Becker and D.H. Gottlieb.
\newblock The transfer map and fiber bundles.
\newblock {\em Topology}, 14:1--12, 1975.

\bibitem{Bott-Taubes}
R.~Bott and C.~Taubes.
\newblock On the self-linking of knots.
\newblock {\em J. Math. Phys.}, 35(10):5247--5287, 1994.

\bibitem{Bredon}
G.E. Bredon.
\newblock {\em Topology and Geometry}.
\newblock Springer-Verlag, 1993.

\bibitem{Budney}
R.~Budney.
\newblock Little cubes and long knots.
\newblock arXiv [math.GT/0309427], to appear in {\it Topology}.

\bibitem{Budney-Cohen}
R.~Budney and F.~R. Cohen.
\newblock On the homology of knot spaces.
\newblock arXiv [math.GT/0504206].

\bibitem{Cattaneo}
A.S. Cattaneo, P.~Cotta-Ramusino, and R.~Longoni.
\newblock {Configuration spaces and Vassiliev classes in any dimension}.
\newblock {\em Algebr. and Geom. Top.}, 2:949--1000, 2002.

\bibitem{FCohen}
F.R. Cohen, T.J. Lada, and J.P. May.
\newblock {\em The homology of iterated loop spaces}, volume 533 of {\em
  Lecture Notes in Math.}
\newblock Springer-Verlag, 1976.

\bibitem{Cohen-Klein}
R.L. Cohen and J.R. Klein.
\newblock Umkehr maps.
\newblock arXiv [math.AT/0711.0540v1].

\bibitem{Fulton-MacPherson}
W.~Fulton and R.~MacPherson.
\newblock A compactification of configuration spaces.
\newblock {\em Ann. of Math., 2nd Ser.}, 139(1):183--225, 1994.

\bibitem{Guadagnini}
E.~Guadagnini, M.~Martellini, and M.~Mintchev.
\newblock {Chern-Simons field theory and link invariants}.
\newblock {\em Nucl. Phys.}, B(330):575--607, 1989.

\bibitem{Laures}
G.~Laures.
\newblock On cobordism of manifolds with corners.
\newblock {\em Trans. Amer. Math. Soc.}, 352(12):5667--5688, 2000.

\bibitem{goils}
J.P. May.
\newblock {\em Geometry of iterated loopspaces}, volume 271 of {\em Lecture
  Notes in Math.}
\newblock Springer-Verlag, 1972.

\bibitem{Thurston}
D.~Thurston.
\newblock {Integral expressions for the Vassiliev knot invariants}.
\newblock A.B. thesis, 1995.
\newblock arXiv [math.QA/9901110].

\bibitem{VolicBT}
I.~Voli\'{c}.
\newblock {A survey of Bott-Taubes integration}.
\newblock arXiv [math.0502295], to appear in {\it J. Knot Theor. Ramif.}

\end{thebibliography}

\end{document}